\documentclass[12pt]{amsart}

\usepackage{comment}
\usepackage[normalem]{ulem}
\usepackage[colorlinks,linkcolor=red,anchorcolor=blue,citecolor=blue]{hyperref}
\usepackage{amscd} 
\usepackage{amsmath}
\usepackage{mathtools}
\usepackage{amssymb} 
\usepackage{amsthm}
\usepackage{bigdelim}
\usepackage{color} 
\usepackage{enumerate}
\usepackage{graphicx}
\usepackage{mathrsfs}
\usepackage{multirow}
\usepackage[all]{xy} 
\usepackage{color} 
\usepackage{enumitem}
\usepackage{ulem}

\newtheorem{thm}{Theorem}[section] 

\newtheorem{cor}[thm]{Corollary}
\newtheorem{prop}[thm]{Proposition}

\newtheorem{lem}[thm]{Lemma}
\theoremstyle{definition} 
\newtheorem{defn}[thm]{Definition}

\newtheorem{ex}[thm]{Example} 
\newtheorem{ques}[thm]{Question}
\theoremstyle{remark}
\newtheorem{rem}[thm]{Remark}
\newtheorem{prob}[thm]{Problem}

\newtheorem{claim}[thm]{Claim}

\numberwithin{equation}{section}

\newcommand{\rk}[0]{\operatorname{rk}}

\newcommand{\codim}[0]{\operatorname{codim}}

\newcommand{\Sym}{{\rm Sym}}

\newcommand{\Hom}[0]{\mathscr{H}\!\textit{om}}

\newcommand{\Alb}{{\rm Alb}}

\newcommand{\Tor}{{\rm tor}}

\newcommand{\reg}{{\rm{reg}}}
\newcommand{\sing}{{\rm{sing}}}
\newcommand{\qt}{{\rm{qt}}}

\newcommand{\deldel}{\sqrt{-1}\partial \overline{\partial}}

\newcommand{\Ker}[1]{\mathrm{Ker}(#1)}

\newcommand{\Z}{\mathbb{Z}}
\newcommand{\N}{\mathbb{Z}_+}
\newcommand{\C}{\mathbb{C}}
\newcommand{\Q}{\mathbb{Q}}

\title[Positivity of tangent sheaves]
{Positivity of tangent sheaves \\ of projective varieties \\
-\,- the structure of MRC fibrations}

\author{Masataka IWAI}
\address{Department of Mathematics, Graduate School of Science, Osaka University,
1-1, Machikaneyama-cho, Toyonaka, Osaka 560-0043, Japan.}
\email{{\tt masataka@math.sci.osaka-u.ac.jp}}
\email{{\tt masataka.math@gmail.com}}

\author{Shin-ichi MATSUMURA}
\address{Mathematical Institute 
$\&$ Division for the Establishment of Frontier Science of Organization for Advanced Studies,  
Tohoku University, 
6-3, Aramaki Aza-Aoba, Aoba-ku, Sendai 980-8578, Japan.}
\email{{\tt mshinichi-math@tohoku.ac.jp}}
\email{{\tt mshinichi0@gmail.com}}

\author{Guolei ZHONG}
\address{Center for Complex Geometry,
	Institute for Basic Science (IBS),
	55 Expo-ro, Yuseong-gu, Daejeon, 34126, Republic of Korea}
\email{{\tt guolei@ibs.re.kr}}
\email{{\tt zhongguolei@u.nus.edu}}



\begin{document}

\begin{abstract}
In this paper, we extend the structure theorems for smooth projective varieties with nef tangent bundle to projective klt varieties whose tangent sheaf is either positively curved or almost nef. 
Specifically, we show that such a variety $X$, up to a finite quasi-\'etale cover, 
admits a rationally connected fibration $X \to A$ 
onto an abelian variety $A$. 
For the proof, we develop the theory of positivity of coherent sheaves on projective varieties.
As applications, we establish some relations between the geometric properties and positivity of tangent sheaves. 
\end{abstract}

\maketitle

\tableofcontents

\newpage

\section{Introduction}\label{Sec-1}
After Mori's solution \cite{Mor79} to the Hartshorne's conjecture, 
it has become clear that a positivity condition of the tangent bundle imposes strong restrictions on the geometry of underlying varieties.
Along this direction, the paper \cite{DPS94} established the structure theorem for smooth projective varieties with nef tangent bundle. 
The nefness of the tangent bundle is viewed as a condition of semi-positive curvature, and in this context,
numerous results on the structure of projective manifolds with semi-positive curvature have been studied  (see \cite{Cao19, CH19, HIM22, Iwa22, Mat20, Mat22}).
Although these structure theorems are not directly linked to the minimal model program, 
interestingly, some of them are generalized to projective varieties with mild singularities arising from the minimal model program 
(see \cite{CCM21, MW21, Mat22}). 

In this paper, building upon this framework, we extend the results of \cite{HIM22, Iwa22}, which were originally formulated for smooth projective varieties, 
to projective varieties with \textit{kawamata log terminal} (klt for short) singularities. 
Specifically, we establish the structure theorems 
for a projective klt variety whose tangent sheaf is positively curved or almost nef. 
We refer to Subsection \ref{subsec-positivity} for the definitions of positively curved or almost nef sheaves. 
The first result (Theorem \ref{thm-positively-curved}) generalizes \cite{HIM22} and the second result (Theorem \ref{thm-almost-nef}) generalizes \cite{Iwa22}.

\begin{thm}\label{thm-positively-curved}
Let $X$ be a projective klt variety with $($possibly singular$)$ positively curved tangent sheaf. 
Then, there exist a finite quasi-\'etale cover $\widehat{X} \to X$ and 
a fibration $\alpha \colon  \widehat{X} \to A$ 
satisfying the following properties$:$
\begin{enumerate}
\item[$(1)$] $\alpha \colon   \widehat{X} \to A$ is a locally constant fibration 
$($in particular, a locally trivial fibration$)$. 
\item[$(2)$] The base variety $A$ is an abelian variety. 
\item[$(3)$] The fiber of $\alpha \colon   \widehat{X} \to A$ is a rationally connected klt variety with positively curved tangent sheaf.  
\end{enumerate}
\end{thm}
The term ``fibration" denotes a surjective proper morphism with connected fibers, 
and we  refer to \cite[Section 2]{MW21} for the definition of locally constant fibrations. 

\begin{thm}\label{thm-almost-nef}
Let $X$ be a projective klt variety with almost nef tangent sheaf. 
Then, there exists a fibration $\alpha \colon  X \to Y$ satisfying the following properties$:$
\begin{enumerate}
\item[$(1)$] The fibration $\alpha \colon  X \to Y$ is  flat, and every irreducible component of the singular locus of $X$ $($if it is non-empty$)$ dominates $Y$. 
\item[$(2)$] The base variety $Y$ is an \'etale quotient of an abelian variety $($i.e.,\,there exists a finite \'etale cover $A \to Y$ from an abelian variety $A$$)$. 

\item[$(3)$] Any fiber of  $\alpha \colon  X \to Y$ is an irreducible and reduced rationally connected klt variety, 
and a very general fiber has almost nef tangent sheaf. 
\end{enumerate}
\end{thm}

The fibration $\alpha \colon  X \to Y$ in Theorem \ref{thm-almost-nef} is not necessarily locally constant, 
but it is expected to be locally trivial (cf.\,\cite[Theorem 1.1]{HIM22}). 
From this perspective, Theorem \ref{thm-almost-nef} (1) is a natural property although it looks technical.
It is known that positively curved \textit{vector bundles} are always almost nef. 
Therefore, when $X$ is smooth,  Theorem \ref{thm-positively-curved} directly follows from Theorem \ref{thm-almost-nef}, except for the local constancy in Theorem \ref{thm-positively-curved} (1). 
However, when $X$ has singularities, the tangent sheaf $\mathcal{T}_{X}$ is not necessarily almost nef even if it is positively curved (see Example \ref{exa-Gachet20}). 
Such a problem comes from the fact that the tangent sheaf is not  locally free. 
Thus, in this paper, we tackle this particular situation by developing the theory of positivity for coherent sheaves  on projective klt varieties.

The proofs of Theorems \ref{thm-positively-curved} and \ref{thm-almost-nef} are essentially different.
The proof of Theorem \ref{thm-positively-curved} is based on the theory of foliations, 
while the proof of Theorem \ref{thm-almost-nef} relies on the functorial resolutions of singularities. 
This difference presents an interesting aspect that never appears in the smooth case.

Our structure theorems allow us to reduce some problems for the ambient space $X$ to the rationally connected fiber. 
As the first application (Corollary \ref{cor-almost-nef}), we show that the almost nefness of the tangent sheaf controls the singularities of $X$,  
which behaves differently from varieties with positively curved tangent sheaf.
Specifically, a quasi-\'etale quotient $X$ of an abelian variety has positively curved tangent sheaf, 
but unless $X$ is smooth, the tangent sheaf $\mathcal{T}_{X}$ cannot be  almost nef.

\begin{cor}\label{cor-almost-nef}
Let $X$ be a projective klt variety with almost nef tangent sheaf. 
If the canonical divisor $K_{X}$ is numerically trivial, 
then $X$ is an \'etale quotient of an abelian variety $($in particular, it is smooth$)$. 
\end{cor}

As the second application, we establish a relation 
between the geometry of $X$ and positivity of the tangent sheaf $\mathcal{T}_X$ by using the Fujita decomposition.
For this purpose, we generalize the Fujita decomposition to reflexive sheaves on projective klt varieties.
Thanks to the Fujita decomposition, as detailed in Theorem \ref{thm-fujita-decomposition-klt}, the tangent sheaf $\mathcal{T}_X$ decomposes into a flat part and a positive part. 
Theorem \ref{thm-MRC-Albanese-main} (1)  shows that the flat part is related to the base variety in the structure theorems.

\begin{thm}
\label{thm-MRC-Albanese-main}
Let \(X\) be a projective klt variety with positively curved or almost nef tangent sheaf.
After we replace $X$ with a maximally quasi-\'etale cover,
there exists a fibration $\alpha \colon  X\to A$ $($as in Theorems \ref{thm-positively-curved} and \ref{thm-almost-nef}$)$ 
onto an abelian variety $A$ of dimension $\widehat{q}(X)$ 
and the pullback $\alpha^{*}\mathcal{T}_{A}$ of the tangent bundle $\mathcal{T}_{A}$  
coincides with the reflexive hull of the flat part of the Fujita decomposition of $\mathcal{T}_{X}$.
In particular, the augmented irregularity $\widehat{q}(X)$ is equal to the rank of the flat part of $\mathcal{T}_{X}$.

\end{thm}

We refer to Section \ref{Sec-3} for the definition of maximally quasi-\'etale covers. 
Based on Theorem \ref{thm-MRC-Albanese-main}, we obtain the following corollary, 
which illustrates a relation among the generic ampleness of the tangent sheaf, the vanishing of augmented irregularity, and the rational connectedness.  

\begin{cor}
\label{cor-maximally-etale-main}
Let $X$ be a projective klt variety with positively curved or almost nef tangent sheaf.
Then, the following conditions are equivalent.
\begin{enumerate}[label=$(\alph*)$]
\item The tangent sheaf $\mathcal{T}_X$ is generically ample.
\item The augmented irregularity $\widehat{q}(X)$ is zero.
\item Any finite quasi-\'etale cover of $X$ is rationally connected.
\end{enumerate}
\end{cor}

Note that there exists a rational klt surface $X$ such that the tangent sheaf $\mathcal{T}_{X}$ is positively curved but not generically ample (see~\cite[Subsection 3.B]{GKP14} and \cite[Example 1.1]{Ou14}). 
As shown in Corollary \ref{cor-maximally-etale-main}, this rational surface admits a non-rational quasi-\'etale cover.
This is a ``pathological'' example that does not occur in the smooth case.  
Theorem \ref{thm-MRC-Albanese-main} and Corollary \ref{cor-maximally-etale-main} provide insights into such ``pathological'' examples, 
especially regarding positivity of the tangent sheaves (see also Question \ref{ques-anf-vanishing-aug}).

This paper is organized as follows. In Section \ref{Sec-2}, we develop a theory of positivity for coherent sheaves on singular varieties. 
In Section \ref{Sec-3}, we generalize the Fujita decomposition (Theorem \ref{thm-fujita-decomposition-klt})  
and a flatness criteria (Proposition \ref{prop-almostnef-c2-locallyfree}) for reflexive sheaves. 
We prove Theorem \ref{thm-positively-curved} in Section \ref{Sec-4} and Theorem \ref{thm-almost-nef} in Section \ref{Sec-5}, 
respectively. 
In Section \ref{Sec-6}, we prove Theorem \ref{thm-MRC-Albanese-main} and Corollary \ref{cor-maximally-etale-main}. 
In Section \ref{Sec-7}, we provide examples of projective varieties whose tangent sheaf is positively curved, almost nef, or satisfies other positivity conditions.

\subsection*{Acknowledgment}\label{subsec-ack}
The authors would like to thank Professors Jie Liu, Niklas M\"uller, and Shou Yoshikawa for valuable comments and suggestions,  
and  to the organizers of the ``Korea-Japan Conference in Algebraic Geometry" 
held in February 2023, where discussions served as a foundation for this paper.
The authors would also like to thank the referee for the careful reading and the suggestions to improve the paper. 

M.\,I.\,was supported by the Grant-in-Aid for Early Career Scientists, $\sharp$22K13907.
S.\,M.\,was supported by the Grant-in-Aid for Scientific Research (B), $\sharp$21H00976, 
the Fostering Joint International Research (A), $\sharp$19KK0342, provided by JSPS, 
and the JST FOREST Program, $\sharp$PMJFR2368, provided by JST.  
G.\,Z.\,was supported by the Institute for Basic Science (IBS-R032-D1-2023-a00).

\section{Preliminaries}\label{Sec-2}
\subsection{Notation and conventions}\label{subsec-notation}

Throughout this paper, we work over the field of complex numbers. 
We employ the terms ``Cartier divisors,"``invertible sheaves," and ``line bundles" interchangeably. 
Similarly, the terms ``locally free sheaves" and ``vector bundles" are used synonymously.
Unless explicitly mentioned, all the sheaves are assumed to be coherent.

Let $\mathcal{E}$ be a torsion-free (coherent) sheaf on a normal variety $X$. 
The notation $X_{\mathcal{E}}\subseteq X$ denotes 
the largest Zariski open subset where $\mathcal{E}$ is locally free. 
Note that $\codim (X \setminus  X_{\mathcal{E}}) \geq 2$ holds. 
We define the dual reflexive sheaf $\mathcal{E}^\vee$ by $\mathcal{E}^\vee\coloneqq \Hom(\mathcal{E}, \mathcal{O}_X)$. 
Furthermore, for a positive integer $m\in\mathbb{Z}_+$, 
we define 
\begin{itemize}
\item the reflexive tensor power $\mathcal{E}^{[\otimes m]}$ by $\mathcal{E}^{[\otimes m]}\coloneqq(\mathcal{E}^{\otimes m})^{\vee\vee}$; 
\item the reflexive symmetric power $\textup{Sym}^{[m]}\mathcal{E}$ by $\textup{Sym}^{[m]}\mathcal{E}\coloneqq(\text{Sym}^m\mathcal{E})^{\vee\vee}$; 
\item  the reflexive exterior power $\wedge^{[m]}\mathcal{E}$ by $\wedge^{[m]}\mathcal{E}\coloneqq(\wedge^m\mathcal{E})^{\vee\vee}$. 
\end{itemize}
The reflexive top exterior power $\det\mathcal{E}\coloneqq \wedge^{[\rk \mathcal{E}]}\mathcal{E}$ 
is called {\textit{the determinant sheaf}} of $\mathcal{E}$, where $\rk \mathcal{E}$ is the rank of $\mathcal{E}$. 
The sheaf \(\mathcal{E}\) of rank one is said to be \textit{\(\mathbb{Q}\)-Cartier} 
if  there exists \(m\in\mathbb{Z}_+\) such that \(\mathcal{E}^{[m]}\) is an invertible sheaf. 
For a torsion-free sheaf $\mathcal{F}$ on \(X\),  
we define the reflexive tensor product $\mathcal{E} [\otimes]\mathcal{F}$  
by $\mathcal{E} [\otimes]\mathcal{F} \coloneqq (\mathcal{E} \otimes\mathcal{F} )^{\vee\vee}$. 
For a morphism $f \colon  Y \rightarrow X$ between normal  varieties, 
we define the reflexive pullback \(f^{[*]}\mathcal{E}\) by $f^{[*]}\mathcal{E} \coloneqq (f^{*}\mathcal{E})^{\vee\vee}$.

Let $\alpha$ be a movable $1$-cycle on $X$ 
(i.e.,\,a codimension one cycle such that 
the intersection number $D\cdot \alpha $ is non-negative for any effective Cartier divisor $D$ on $X$). 
The {\it slope of $\mathcal{E}$ with respect to $\alpha$} is defined by 
$$
\mu_{\alpha }(\mathcal{E}) \coloneqq \frac{ c_1(\mathcal{E})\cdot\alpha}{{\rk } \mathcal{E}}.
$$
The maximum slope $\mu_{\alpha }^{\max}(\mathcal{E})$ is defined by the supremum of $\mu_{\alpha }(\mathcal{F})$ where \(\mathcal{F}\) runs over all the non-zero  subsheaves $\mathcal{F} \subseteq \mathcal{E}$ and the minimum slope 
$\mu_{\alpha }^{\min}(\mathcal{E})$ is defined by the infimum of $\mu_{\alpha }(\mathcal{Q})$ where \(\mathcal{Q}\) runs over all the torsion-free quotients  $\mathcal{E} \rightarrow \mathcal{Q}$. 
The sheaf $\mathcal{E} $ is said to be $\alpha$-{\it semistable} 
if $\mu_{\alpha }^{\max}(\mathcal{E})=\mu_{\alpha }(\mathcal{E})$.
By the Harder-Narasimhan filtration,  
there exists the unique saturated semi-stable subsheaf $\mathcal{E}_{\max} \subseteq \mathcal{E}$, 
called the {\it maximal destabilizing subsheaf} of $\mathcal{E}$, such that 
$\mu_{\alpha }(\mathcal{E}_{\max}) = \mu_{\alpha }^{\max}(\mathcal{E})$.

The tangent sheaf $\mathcal{T}_X$ of $X$ is defined by the reflexive hull $\mathcal{T}_X\coloneqq (j_{*} \mathcal{T}_{X_{\reg}})^{\vee \vee}$ 
of the direct image sheaf $j_{*}  \mathcal{T}_{X_{\reg}}$, 
where $j\colon  X_{\reg} \to X$ is the natural inclusion and $\mathcal{T}_{X_{\reg}}$ is the tangent bundle on $X_{\reg}$. 
Similarly, the reflexive cotangent sheaves of degree \(p\) is defined by $\Omega_{X}^{[p]}\coloneqq (j_{*} \Omega_{X_{\reg}}^{p})^{\vee \vee}$.

A \textit{foliation} on  \(X\) is a saturated  subsheaf \(\mathcal{F}\subseteq \mathcal{T}_X\) 
(i.e., the quotient \(\mathcal{T}_X/\mathcal{F}\) is torsion-free) such that  \(\mathcal{F}\) is closed under the Lie bracket. 
The rank \(r\) of the foliation \(\mathcal{F}\) is defined by that of \(\mathcal{F}\) and the codimension of \(\mathcal{F}\) is defined by \(q\coloneqq \dim X-r\). 
The canonical class \(K_{\mathcal{F}}\) of \(\mathcal{F}\) is a Weil divisor $-K_{\mathcal{F}}$ on \(X\) 
such that \(\mathcal{O}_X(-K_{\mathcal{F}})\cong\det\mathcal{F}\). 

Let \(X^\circ\subseteq X_{\text{reg}}\) be the open set where \(\mathcal{F}|_{X_{\text{reg}}}\) is a subbundle of \(\mathcal{T}_{X_{\text{reg}}}\).
A \textit{leaf} of \(\mathcal{F}\) is a maximal connected and immersed holomorphic submanifold \(L\subseteq X^\circ\) such that \(\mathcal{T}_L=\mathcal{F}|_L\).
The foliation \(\mathcal{F}\) is said to be \textit{algebraically integrable} if every leaf passing through a general point is algebraic 
(i.e.,\,the leaf is open in its Zariski closure). 
The \(r\)-th wedge product gives rise to a non-zero morphism \(\mathcal{O}_X(-K_{\mathcal{F}})\to (\wedge^{[r]}\mathcal{T}_X)\). 
The foliation \(\mathcal{F}\) is said to be \textit{weakly regular} 
if the induced dual morphism \(\Omega_{X}^{[r]} [\otimes ] \mathcal{O}_{X} (-K_{\mathcal{F}}) \to \mathcal{O}_{X}\)  
is a surjective sheaf morphism (see \cite[Section 5]{Dru21}).
An algebraically integrable and weakly regular foliation determines an equidimensional fibration by \cite[Theorem 17]{DGP20}, 
which plays a crucial role in obtaining an everywhere-defined MRC fibration in the proof of Theorem \ref{thm-positively-curved}.

\subsection{Positivity of coherent sheaves on projective varieties}\label{subsec-positivity}

In this subsection, we review several notions of positivity defined for torsion-free sheaves and establish their key properties.

\begin{defn}\label{defn-posi}
Let  $X$ be a normal projective variety of dimension $n$ and  $\mathcal{E}$ be a  sheaf on $X$. 

\begin{enumerate}
\item[$(1)$]  $\mathcal{E}$ is said to be {\it ample} (resp.\,{\it strictly nef, nef}) if the tautological line bundle $\mathcal{O}_{\mathbb{P}_X(\mathcal{E})}(1)$ is ample (resp.\,strictly nef, nef) on the projectivization $\mathbb{P}_X(\mathcal{E})\coloneqq \text{Proj}(\text{Sym}^{\bullet}\mathcal{E})$. 

\item[$(2)$]$\mathcal{E}$ is said to be {\it almost nef} 
if there exist countably many proper subvarieties $Z_{i} \subseteq X$  ensuring that
the sheaf $\mathcal{E} |_{C} \coloneqq \mathcal{E} \otimes \mathcal{O}_C$ is nef 
for any  curve $ C \not \subseteq \cup_{i} Z_i$ (see \cite[Definition 6.4]{DPS01} and \cite[Definition 3.6]{LOY20}). 
The notation $\mathbb{S}(\mathcal{E})$ denotes the smallest countable union $\cup_{i} Z_i$ with the above property. 

\end{enumerate}

Henceforth, we further assume that $\mathcal{E}$ is torsion-free. 
\begin{enumerate}
\item[$(3)$] $\mathcal{E}$ is said to be {\it pseudo-effective}  if
 for any $a \in \mathbb{Z}_+$ and any ample Cartier divisor $A$ on $X$, there exists 
$b \in \mathbb{Z}_+$ such that $\Sym^{[ab]}( \mathcal{E}  ) \otimes A^{b}$ is globally generated at a general point of $X$ (see \cite[Definition 3.20]{Nak04} and \cite[Definition 7.1]{BDPP13}; cf.~\cite{Vie83}).

\item[$(4)$]  $\mathcal{E}$ is said to be \textit{generically ample}  (resp.\,{\it generically nef}) if $\mu_{H_1\cdots H_{n-1}}^{\text{min}}(\mathcal{E}) >0$ (resp.\,$\ge 0$) holds for any ample Cartier divisors $H_1, \ldots, H_{n-1}$ on $X$
 (see \cite{Miy87}).

\item[$(5)$]   $\mathcal{E}$ is said to be \textit{positively curved} if $\mathcal{E}$ admits a $($possibly singular$)$ Hermitian metric $h$ such that 
$\log |u|_{h^{\vee}}$ is a psh function 
for any local section $u$ of $\mathcal{E}^{\vee}$, 
where $h^{\vee}$ is the dual metric defined by $h^{\vee}  \coloneqq {}^t\! h^{-1}$ (see \cite[Definition 19.1]{HPS18}, \cite[Definition 2.2.2]{PT18}). 
 \end{enumerate}
 \end{defn}
 
Our definition of the pseudo-effectivity of $\mathcal{E}$ differs from 
the pseudo-effectivity of $\mathcal{O}_{\mathbb{P}_X(\mathcal{E})}(1)$  even when $\mathcal{E}$ is locally free 
(see  \cite[Proposition 2.4]{Mat22} for the details and characterizations). 
Note that  $\mathcal{E}$ in (1) and (2) is not assumed to be torsion-free  
and $\mathcal{E} |_{C} \coloneqq \mathcal{E} \otimes \mathcal{O}_C$ in (2) may have a torsion. 

For torsion-free sheaves, we can summarize  relations among the above-mentioned notions by the following table:
\[
\xymatrix@C=30pt@R=25pt{
  \txt{ample} \ar@{=>}[r]\ar@{=>}[dd]&   \txt{strictly nef} \ar@{=>}[r] &\txt{nef}\ar@{=>}[r]\ar@{=>}[d]  &\txt{almost nef}\ar@{=>}[dd] \\
&   \text{positively curved} \ar@{=>}[r]& \text{pseudo-effective} \ar@{=>}[rd]& \\ 
\text{generically ample }  \ar@{=>}[rrr]&& & \text{generically nef}\\
}
\]
We now prove the non-trivial implications in the table.

\begin{prop}\label{pro-basic-implication}
Let \(\mathcal{E}\) be a torsion-free sheaf on a normal projective variety \(X\).
Then, we have$:$
\begin{enumerate}
\item[$(1)$] If \(\mathcal{E}\) is nef, then \(\mathcal{E}\)  is pseudo-effective.
\item[$(2)$]If \(\mathcal{E}\) is  pseudo-effective, then \(\mathcal{E}\)  is generically nef.
\item[$(3)$] If \(\mathcal{E}\) is almost nef, then \(\mathcal{E}\)  is  generically nef.
\item[$(4)$] If \(\mathcal{E}\) is positively curved, then \(\mathcal{E}\)  is  pseudo-effective.
\end{enumerate}
\end{prop}

\begin{proof}
(1), (4). Conclusions (1), (4) are direct consequences of \cite[Proposition 2.4 (4), (2)]{Mat22}, respectively.

\smallskip 
(2). 
Let $\{H_{i}\}_{i=1}^{n-1}$ be  ample Cartier divisors on $X$. 
It is sufficient to show that an arbitrary torsion-free quotient \(\mathcal{E}\to\mathcal{Q} \to 0\) 
satisfies that \(c_1(\mathcal{Q})\cdot H_1\cdots H_{n-1}\ge 0\). 
By  \cite[Lemma 1.4]{Vie83}, 
the determinant sheaf and any torsion-free quotient of pseudo-effective sheaves are pseudo-effective. 
This indicates that $\det(\mathcal{Q})$ is a pseudo-effective sheaf. 
Hence, by \cite[Proposition 2.4 (1)]{Mat22}, there exists an ample Cartier divisor $A$ on $X$ such that 
$(\det \mathcal{Q})^{[m]}\otimes A$ is generically globally generated for any $m \in \N$. 
This implies that $c_{1}(\mathcal{Q})$ is pseudo-effective, 
and thus we have \(c_{1}(\mathcal{Q})\cdot H_1\cdots H_{n-1}\ge 0\).

\smallskip 
(3). 
Let \(\mathcal{E}\to\mathcal{Q} \to 0\) be a torsion-free quotient. 
Then, the quotient sheaf  $\mathcal{Q} $ is almost nef by Lemma \ref{lem-elementary-lem-almostnef} (4). 
We may assume that $H_{i}$ is a very general hypersurface in the linear system of a very ample Cartier divisor. 
Then, the curve \(C=H_1 \cap \cdots\cap H_{n-1}\) does not intersect with $X \setminus (X_{\reg} \cap X_{\mathcal{E}})$ 
or  $\mathbb{S}(\mathcal{E})$. 
Hence, the restriction $\mathcal{Q}|_{C} $ is a nef vector bundle, 
and we obtain \(c_{1}(\mathcal{Q})\cdot H_1\cdots H_{n-1}= c_1(\mathcal{Q}|_C) \ge 0\). 
\end{proof}
We confirm that the converse implications in the table do not hold.  
\begin{rem}
(1) 
A generically nef vector bundle is not necessarily pseudo-effective nor almost nef. 
Indeed, the tangent bundle of any K3 surface (known to be stable) is generically nef (see \cite[Theorem 3.1]{Miy87}), 
but is not almost nef and not pseudo-effective  (see \cite[Theorem 7.7]{BDPP13} and \cite[Theorem 1.6]{HP19}).

\smallskip 
(2)  
A pseudo-effective vector bundle is almost nef, but this does not hold for torsion-free sheaves. 
The first assertion follows from \cite[Theorem 6.2]{BDPP13}. 
The second assertion is first confirmed by Example \ref{exa-Gachet20}, 
and can also be understood more uniformly by Corollary \ref{ques-high-dim-not-almost-but=pseudo-effective}.

\smallskip 
(3)  An almost nef vector bundle is not necessarily pseudo-effective. 
Indeed, there exists an almost nef  vector bundle on a Hirzebruch surface that is not pseudo-effective (see \cite{EFI23}). 
\end{rem}

\subsection{Fundamental properties of almost nef sheaves}\label{subsec-almostnef}
In this subsection, we establish the key properties of almost nef sheaves. 
Let us begin with the nefness of sheaves.

\begin{lem}\label{lem-elementary-pro-nef}
Let $\mathcal{E}$ and \(\mathcal{F}\) be sheaves on a normal projective variety \(X\). 
Then, we have the following assertions. 
\begin{enumerate}
\item[$(1)$] The sheaf $\mathcal{E}$ is nef $($resp.\,strictly nef$)$ if and only if for any finite morphism $ \nu \colon  C \to X$  from any smooth curve $C$, the pullback $\nu^{*}\mathcal{E}$ is nef 
$($resp.\,strictly nef$)$. 

\item[$(2)$]  Let $f \colon  Y \to X$ be a morphism between normal projective varieties.
\begin{enumerate}
\item[$\bullet$] If $\mathcal{E}$ is nef, then $f^{*}\mathcal{E}$ is nef.  
\item[$\bullet$] If $f \colon  Y \to X$ is finite and $\mathcal{E}$ is  strictly nef, then  $f^{*}\mathcal{E}$ is strictly nef.
\item[$\bullet$] We further assume that $f \colon  Y \to X$ is surjective.
If $f^{*}\mathcal{E}$ is nef $($resp.\,strictly nef$)$, then $\mathcal{E}$ is nef $($resp.\,strictly nef$)$.
\end{enumerate}

\item[$(3)$]  
Let $\mathcal{E} \to \mathcal{F}$ be a surjective sheaf morphism. 
If  $\mathcal{E}$ is nef, then $\mathcal{F}$ is nef. 
If $X$ is an irreducible curve, then the same statement holds 
under the weaker assumption that $\mathcal{E} \to \mathcal{F}$ is generically surjective.
In particular, if \(\mathcal{E}\) is a nef sheaf on a smooth curve,
 then the locally free sheaf $\mathcal{E}/\Tor$ is nef.
Here $\mathcal{E}/\Tor$ is the quotient sheaf of $\mathcal{E}$ with its torsion subsheaf.

\item[$(4)$] 
If $X$ is a smooth curve and $\mathcal{E}^{\vee\vee}=\mathcal{E}/ \Tor$ is nef, 
then $\mathcal{E}$ is nef.

\end{enumerate}
\end{lem}
\begin{proof}
Conclusions (1), (3) follow from \cite[Proposition 3.2, Corollary 3.3 and Corollary 3.5]{LOY20}, 
and Conclusion (2) follows from \cite[Lemma 6.2.8 (ii) and Theorem 6.2.12 (i)]{Laz04} and \cite[Proposition 2.2]{LiOY18}, 
by noting that the arguments in the smooth case still work in the singular case. 

By \cite[Proposition 3.2]{LOY20}, 
it is sufficient for (4) to show that  a quotient line bundle $\mathcal{E} \to \mathcal{L}$ satisfies that 
$\deg(\mathcal{L}) \ge 0 $. 
Note that any torsion-free sheaves on $X$ are locally free since $X$ is a smooth curve. 
By the universal property of reflexive hull, 
the morphism $\mathcal{E} \to \mathcal{L}$ 
factors through $\mathcal{E} \to \mathcal{E}^{\vee\vee} \to \mathcal{L}$.
Since $\mathcal{E}^{\vee\vee} \to \mathcal{L}$ is surjective, 
we obtain $\deg(\mathcal{L}) \ge 0 $ by (3). 
\end{proof}

Lemma \ref{lem-elementary-pro-nef} (4) is not true unless $X$ is a smooth curve by the following example. 

\begin{ex}\label{exa-ideal-sheaf-blown-up}
Let \(X\) be a smooth projective surface and \(\mathcal{I}\) an ideal sheaf at one point.  
The sheaf $\mathcal{I}$ is torsion-free and the reflexive hull $\mathcal{I}^{\vee\vee}=\mathcal{O}_{X}$ is nef. 
However, the sheaf $\mathcal{I}$ itself is not nef (and  even not almost nef).
Indeed, taking the blow up $\pi \colon  \widetilde{X} \to X$ along $\mathcal{I}$,  
we obtain $\pi^{*} \mathcal{I}/\Tor = \mathcal{O}_{\widetilde{X}}(- E)$, where $E$ is the \(\pi\)-exceptional divisor.
If $\mathcal{I}$ is nef (or almost nef),  so is  $\pi^{*} \mathcal{I}/\Tor= \mathcal{O}_{\widetilde{X}}(- E)$ 
by Lemma \ref{lem-elementary-pro-nef} (2). 
This is a contradiction. 

This example serves as a comparison with other types of positivity. 
The ideal sheaf $\mathcal{I}$ is pseudo-effective, but its reflexive pullback $\pi^{[*]} \mathcal{I} = \pi^{*} \mathcal{I}/\Tor$ is not pseudo-effective. 
Hence, the reflexive pullback of a pseudo-effective torsion-free sheaf is not necessarily pseudo-effective. 
Furthermore, the ideal sheaf $\mathcal{I}$ is not strongly pseudo-effective in the sense of Wu (see \cite[Remark 4]{Wu22}). 
\end{ex}

At the end of this subsection, we prove the key properties of almost nef sheaves. 
Some properties in Lemma \ref{lem-elementary-lem-almostnef} are not expected for pseudo-effective sheaves 
(even for positively curved sheaves). 
For this reason, Theorem \ref{thm-positively-curved} requires a different proof from that of Theorem \ref{thm-almost-nef}. 
For example, Lemma \ref{lem-elementary-lem-almostnef} (6) clearly does not hold for positively curved sheaves 
by Example \ref{exa-ideal-sheaf-blown-up}. 

\begin{lem}
\label{lem-elementary-lem-almostnef} 
Let $\mathcal{E}$ be a $($not necessarily torsion-free$)$ sheaf on a normal projective variety \(X\). 
Then, we have$:$
\begin{enumerate}

\item[$(1)$]  $\mathcal{E}$ is almost nef if and only if 
there exist  countably many proper subvarieties $Z_{i} \subseteq X$ such that 
$\nu^{*}\mathcal{E}$ is nef for any finite morphism $\nu \colon  C \to X$, 
where $C$ is a smooth curve with $ C \not \subseteq \cup_{i} Z_i$.

\item[$(2)$] Let $f \colon  Y \to X$ be a surjective morphism. 
Then, the sheaf $\mathcal{E}$ is almost nef if and only if $f^{*}\mathcal{E}$ is almost nef.

\item[$(3)$] If $\mathcal{E}$ is almost nef, then $\Sym^{[m]}\mathcal{E}$ and $\wedge^{[m]} \mathcal{E}$ are almost nef. 

\item[$(4)$] If $\mathcal{E} \to \mathcal{F}$ is a generically surjective sheaf morphism 
and $\mathcal{E}$ is almost nef, then so is $\mathcal{F}$.
In particular, together with $(2)$, if $\mathcal{E}$ is almost nef, 
then the reflexive pullback $f^{[*]} \mathcal{E}$ is almost nef for any morphism  $f \colon  Y \to X$.

\item[$(5)$] If \(\mathcal{E}\) and \(\mathcal{F}\) are almost nef, then \(\mathcal{E}\otimes\mathcal{F}\) is almost nef.

\item[$(6)$] 
If \(\mathcal{E}\) is almost nef, then any non-zero sheaf morphism \(\mathcal{E}\to \mathcal{O}_{X}\) is surjective.

\item[$(7)$] Let \(\mathcal{E}\to \mathcal{F}\) be a surjective morphism such that 
 $\mathcal{F}$ is locally free and \(\textup{rk}\,\mathcal{E}=\textup{rk}\,\mathcal{F}\) holds. 
If \(\mathcal{F}\) is almost nef, then \(\mathcal{E}\) is almost nef.

\item[$(8)$]
Assume that \(\mathcal{E}\) is torsion-free. 
If \(f \colon Y\to X\) is a  surjective morphism such that \(f^*\mathcal{E}/\Tor\) is locally free and almost nef,
then \(\mathcal{E}\) is almost nef.
\end{enumerate}
\end{lem}
\begin{proof}
(1). The part of ``only if''  is clear from definition. 
To check the part of  ``if,'' 
we take any curve $ C$ with $C \not \subseteq \cup_{i} Z_i$. 
Then, for the normalization \(\nu \colon \overline{C}\to C\), 
the pullback  \(\nu^*(\mathcal{E}|_C)\) is nef by  assumption, 
which indicates that \(\mathcal{E}|_C\) is nef.

\smallskip
(2).
Assume that \(\mathcal{E}\) is almost nef. 
Take a finite morphism $\nu_Y \colon  C_Y \to Y$ with a smooth curve $C_{Y}$ 
such that $\nu_Y(C_Y)$ is not contained in $ f^{-1}(\mathbb{S}(\mathcal{E}))$, 
by noting that $ f^{-1}(\mathbb{S}(\mathcal{E}))$ is still a countable union of proper subvarieties.  
We may assume that $C_X\coloneqq f \circ \nu_{Y}(C_{Y})$ is not point. 
Since $\mathcal{E}|_{C_X}$ is nef, 
the pullback $\nu_Y^{*}(f^*\mathcal{E})= (f|_{C_{X}} \circ \nu_{Y})^{*} (\mathcal{E}|_{C_X})$ is nef. 
This indicates that  \(f^*\mathcal{E}\) is almost nef by (1).

Assume that \(f^*\mathcal{E}\) is almost nef. 
Let \(\mathbb{S}_1(f^*\mathcal{E})\subseteq \mathbb{S}(f^*\mathcal{E})\) 
be the countable union of subvarieties in  $\mathbb{S}(f^*\mathcal{E})$ that does not dominate \(X\).
Let \(\mathbb{S}(\mathcal{E})\) be the union of the image \(f(\mathbb{S}_1(f^*\mathcal{E}))\) and the non-flat locus 
of \(f \colon  Y \to X\). 
Note that the non-flat locus is a proper subvariety by the generic flatness. 

For a curve \(C\) with $ C  \not \subseteq \mathbb{S}(\mathcal{E})$, 
we first claim that  \(f^{-1}(C)\not\subseteq \mathbb{S}(f^*\mathcal{E})\).
By our choice of \(C\), we have \(f^{-1}(C)\not\subseteq \mathbb{S}_1(f^*\mathcal{E})\).
Suppose that \(f^{-1}(C)\subseteq \mathbb{S}(f^*\mathcal{E})\backslash\mathbb{S}_1(f^*\mathcal{E})\).
Take an irreducible component \(T\) of \(\mathbb{S}(f^*\mathcal{E})\backslash\mathbb{S}_1(f^*\mathcal{E})\) containing \(f^{-1}(C)\).
Then, by the definition of $\mathbb{S}_1(f^*\mathcal{E})$, the variety  \(T\) dominates \(X\).
Meanwhile, for a general point \(c\in C\), the generic flatness shows that 
\[
\dim T -\dim X=\dim T_c =\dim Y_c =\dim Y-\dim X. 
\]
This indicates \(\dim T=\dim Y\), which is a contradiction. 
Our claim is thus proved.

By the claim, we can take a horizontal curve \(C_Y\subseteq f^{-1}(C)\) such that 
$C \not \subseteq \mathbb{S}(f^*\mathcal{E})$. 
Since  \(f^*\mathcal{E}|_{C_Y}\) is nef, 
so is $\mathcal{E}|_{C}$. 
Using Lemma \ref{lem-elementary-pro-nef} (2), we complete the proof of (2).

\smallskip
(3), (4). Conclusions (3) and (4) follow from \cite[Proposition 3.8]{LOY20}.

\smallskip
(5).  
Let  \(C\) be a curve  such that \( C \not \subseteq \mathbb{S}(\mathcal{E})\cup\mathbb{S}(\mathcal{F})\), 
and let \(\nu \colon \overline{C}\to C\) be the normalization. 
By Lemma \ref{lem-elementary-pro-nef} (4), it is sufficient to show that 
$$
(\nu^*\mathcal{E}\otimes\nu^*\mathcal{F})^{\vee \vee} = (\nu^*\mathcal{E}\otimes\nu^*\mathcal{F})/\Tor 
=\nu^*\mathcal{E}/\Tor\otimes \nu^*\mathcal{F}/\Tor 
$$ 
is nef. 
By Lemma \ref{lem-elementary-pro-nef} (3), 
both $\nu^*\mathcal{E}/\Tor$ and $\nu^*\mathcal{F}/\Tor$ are nef. 
Consequently, since the tensor product of nef vector bundles is still nef, 
their tensor product \(\nu^*\mathcal{E}/\Tor\otimes \nu^*\mathcal{F}/\Tor\) is nef. 
This finishes the proof of (5).

\smallskip
(6). The conclusion is obvious when $X$ is a smooth curve since proper ideal sheaves on smooth curves are not nef. 
We will reduce the proof to the case where $X$ is a smooth curve. 
Let \(\mathcal{L}\) be the image of the non-zero sheaf morphism \(\mathcal{E}\to\mathcal{O}_{X}\).
Suppose that \(0\neq\mathcal{L}\subseteq\mathcal{O}_{X}\) is a proper ideal sheaf.
Consider the  exact sequence
\[
0\to \mathcal{L}\to\mathcal{O}_{X}\to \mathcal{O}_{X}/\mathcal{L}\to 0.
\]
Let \(C\) be a smooth curve such that $C \not \subseteq \mathbb{S}(\mathcal{E})$ holds and 
$C $ intersects with the support of the quotient sheaf $\mathcal{O}_{X}/\mathcal{L}$. 
Since the functor of the tensor product is right exact,  we obtain 
\[
\mathcal{L}|_C\to\mathcal{O}_C\to (\mathcal{O}_{X}/\mathcal{L})|_C = (\mathcal{O}_{C}/\mathcal{L}|_C)\to 0.
\]
Note that the left exactness might break down, so \(\mathcal{L}|_C\) is not necessarily an ideal sheaf of $\mathcal{O}_{C}$. 
Then, the condition of  \(\mathcal{O}_{C}/\mathcal{L}|_C\neq 0\) shows that  
the image \(\mathcal{L}\cdot\mathcal{O}_C \subseteq\mathcal{O}_C\) is a proper ideal sheaf of \(\mathcal{O}_C\).  
By the choice of $C$, the sheaf $\mathcal{E}|_{C}$ is nef, and so is $\mathcal{L}|_{C}$ (cf.~Lemma \ref{lem-elementary-pro-nef} (3)). 
By the surjectivity of the map \(\mathcal{L}|_C\to \mathcal{L}\cdot\mathcal{O}_C\), 
the proper ideal sheaf $\mathcal{L}\cdot\mathcal{O}_C$ is also nef, 
but this contradicts the fact in the case where $X$ is a smooth curve. 

\smallskip
(7). 
Let \(C\) be a curve such that \( C \not \subseteq \mathbb{S}(\mathcal{F}) \cup (X \setminus (X_{\reg} \cap X_{\mathcal{E}}))\).
By the right exactness of tensor products,
the induced sheaf morphism \(\mathcal{E}|_C\to \mathcal{F}|_C\) is surjective. 
Meanwhile, by the assumption of the rank, 
the sheaf morphism \(\mathcal{E}\to \mathcal{F}\) on $X_{\reg} \cap X_{\mathcal{E}}$ 
is an isomorphism on a (non-empty) Zariski open set. 
Hence, the induced sheaf morphism \(\mathcal{E}|_C\to \mathcal{F}|_C\) is generically injective.  
Then, we can conclude that \(\mathcal{F}|_C=(\mathcal{E}|_C)^{\vee\vee}\).
This implies that $\mathcal{E}|_C$ is nef  by Lemma \ref{lem-elementary-pro-nef} (5).

\smallskip
(8). 
The sheaf \(f^*\mathcal{E}\) is almost nef by Lemma \ref{lem-elementary-lem-almostnef} (7),   
and thus \(\mathcal{E}\) is almost nef by  Lemma \ref{lem-elementary-lem-almostnef} (2). 
\end{proof}

As the following example shows, the almost nefness is not well behaved by the reflexive pullback. 
This example will be generalized to the higher dimensional case in the proof of Corollary \ref{ques-high-dim-not-almost-but=pseudo-effective}.

\begin{ex}[{\cite[Remark 2.6]{Gac20}}]
\label{exa-Gachet20} 

Let $A$ be an abelian surface and let $i \colon  A \to A$ be 
the involution sending  \(a \in A\) to \(-a \in A\). 
Consider the natural quotient $\pi \colon  A \to X\coloneqq A/\left<i\right>$. 
Since \(\pi\) is a quasi-\'etale cover, 
the reflexive pullback $\pi^{[*]}\mathcal{T}_{X} =\mathcal{T}_A \cong \mathcal{O}_{A}^{\oplus 2}$ is nef.  
Meanwhile, as calculated in \cite[Remark 2.6]{Gac20}, 
there exists a divisorial sheaf $\mathcal{F}$ such that $\mathcal{T}_X \cong \mathcal{F}^{\oplus 2}$ and $\mathcal{F}^{\otimes 2}/\Tor \cong \mathcal{I}_{\text{Sing}(X)}$.
Then, the tangent sheaf $\mathcal{T}_{X}$ is not almost nef 
since \(\mathcal{T}_{X}|_C\) is not nef for a curve \(C\) intersecting with $X_{\sing}$. 
Note that the tangent sheaf $\mathcal{T}_{X}$ is pseudo-effective 
since $\mathcal{F}^{[2]} \cong \mathcal{O}_{X}$. 
This example also shows that pseudo-effective sheaves are not necessarily almost nef in the non-locally free case.
\end{ex}

At the end of this subsection, we compare the positivity of (Weil) divisors to the associated sheaf.

\begin{lem}\label{lem-weil-ref- pseudo-effective }
Let \(D\) be a Weil divisor on a normal projective variety \(X\), 
and let \(\mathcal{E}\coloneqq \mathcal{O}_X(D)\) be the divisorial sheaf associated with $D$. 
Then, we have$:$
\begin{enumerate}
\item[$(1)$] If \(\mathcal{E}=\mathcal{O}_X(D)\) is pseudo-effective  or almost nef 
in the sense of Definition \ref{defn-posi}, 
then  \(D\) is a pseudo-effective \(\mathbb{R}\)-Weil divisor $($cf.~\cite[Definition 2.4]{MZ18}$)$.

\item[$(2)$] Conversely,  if \(D\) is a pseudo-effective \(\mathbb{R}\)-Weil divisor, 
then \(\mathcal{E}=\mathcal{O}_X(D)\) is pseudo-effective, but not necessarily almost nef. 

\item[$(3)$] We further assume that $D$ is  \(\mathbb{Q}\)-Cartier. 
If \(\mathcal{E}=\mathcal{O}_X(D)\)  is almost nef, then \(D\) is pseudo-effective as $\mathbb{Q}$-Cartier divisors. 

\end{enumerate}
\end{lem}
\begin{proof}
(1). 
Suppose first that \(\mathcal{E}\) is pseudo-effective. 
Note that by \cite[Proposition 2.4]{Mat22}, the divisorial sheaf \(\mathcal{E}\) is pseudo-effective 
if and only if there exists an ample Cartier divisor \(A\) on \(X\) such that for any \(m\in\mathbb{Z}_+\), 
the sheaf \(\mathcal{O}_X(mD+A)=\text{Sym}^{[m]}\mathcal{E}\otimes A\) is generically globally generated.
Then, for a non-zero global section \(s_m\in H^0(X,\mathcal{O}_X(mD+A))\), 
we can associate the effective divisor linearly equivalent to \(mD+A\).
Hence, the Weil divisor \(D\) lies in the closure of the set of classes of effective codimension one cycles with real coefficients, 
and thus $D$ is a pseudo-effective \(\mathbb{R}\)-Weil divisor.

Suppose second that \(\mathcal{E}\) is almost nef.
Let us consider the construction in \cite{Nak04} (cf.~\cite[Definition 2.2]{HP19} and \cite[Setup 2.3]{Mat22}).
Let \(H\) be any ample Cartier divisor on \(X\).
Since the reflexive sheaf \(\mathcal{E}\) is almost nef, by \cite[Lemma 2.3]{HP19}, for any \(c>0\), there exists \(i_c,j_c\in \N\) such that \(i_c>cj_c\) and
\[
H^0(X,\text{Sym}^{[i_c]}\mathcal{E}\otimes H^{j_c})=H^0(X,\mathcal{O}_X(i_cD+j_cH))\neq 0.
\]
This indicates that\(D+(j_c/i_c)H\) is \(\mathbb{Q}\)-linearly equivalent to an effective Weil divisor, 
by noting that \(D+(j_c/i_c)H\) is not necessarily \(\mathbb{Q}\)-Cartier.
As the ratio \(j_c/i_c\) goes to zero whenever \(c\) tends to infinity, we see that \(D\) is a pseudo-effective  \(\mathbb{R}\)-Weil divisor.

\smallskip
(2). Suppose that $D$ is a pseudo-effective codimension one cycle. 
Then,  for an ample Cartier divisor \(A\),  the divisor \(mD+A\) is a big Weil divisor for any  \(m \in \mathbb{Z}_{+}\) 
(see \cite[Theorem 3.5]{FKL16}).
In particular,  there exists an integer   \(k \in \mathbb{Z}_{+}\)  such that \(h^0(X,\mathcal{O}_X(k(mD+A))) >0\), 
which implies that \(\text{Sym}^{[km]}\mathcal{E}\otimes\mathcal{O}_X(kA)\) is generically globally generated.
By \cite[Proposition 2.4 (6)]{Mat22}, the sheaf \(\mathcal{E}\) is pseudo-effective. 

Note that the sheaf \(\mathcal{F}\)  in Example \ref{exa-Gachet20} is not almost nef, but 
the Weil divisor $D$ with \(\mathcal{F}=\mathcal{O}_{X}(D)\)  is numerically trivial (in particular, nef).

\smallskip
(3).
This is a consequence of   \cite{BDPP13}. 
Indeed, by assumption, the divisor $mD$ can be regarded as a line bundle  for some $m \in \N$. 
For a resolution $\pi \colon  \widetilde{X} \to X$ of singularities, 
the pullback $\pi^{*}(\mathcal{O}_{X}(mD))$ is an almost nef line bundle, 
and hence $\pi^{*}(\mathcal{O}_{X}(mD))$ is pseudo-effective by \cite{BDPP13}. 
This indicates that the $\mathbb{Q}$-Cartier divisor $D$ is  pseudo-effective. 
\end{proof}

\subsection{Preliminary results}\label{subsec-pre}

In this subsection, we review several propositions which will be frequently used in this paper. 
Although the contents herein may be known to experts, we provide proofs and references for the sake of clarity. 

The following lemma is a slight generalization of \cite[Proposition 1.16]{DPS94} to pseudo-effective or almost nef sheaves. 

\begin{lem}\label{lem-nonvanish}
Let $X$ be a normal projective variety and let 
$\mathcal{E}$ be a pseudo-effective or almost nef torsion-free  sheaf on $X$.

\begin{itemize}
\item[$(1)$]
Any non-zero section $\tau \in H^{0} (X, \mathcal{E}^{\vee})$ is nonvanishing on $X_{\mathcal{E}} \cap X_{\reg}$.

\item[$(2)$] Let $\mathcal{S}$ be a reflexive  sheaf such that $\det \mathcal{S} [\otimes] \mathcal{E}$ 
is a pseudo-effective or almost nef sheaf. 
Consider an injective sheaf morphism $0 \to \mathcal{S} \to \mathcal{E}^\vee$. 
Then, the sheaf $\mathcal{S}$ is locally free on $X_\mathcal{E}\cap X_{\reg}$ and 
the above morphism is an injective bundle morphism on $X_\mathcal{E}\cap X_{\reg}$. 
\end{itemize}
\end{lem}

\begin{proof}
The condition of $\tau $ being nonvanishing at $p \in X_{\mathcal{E}}$ 
simply means that the value $\tau(p) \in \mathcal{E} \otimes \mathbb{C}_{p}  \cong \mathbb{C}^{\rk \mathcal{E}}$ is 
a non-zero vector.
Conclusion (2) is a direct consequence of conclusion (1) and \cite[Lemma 1.20]{DPS94} (cf.\,\cite[Lemma 3.1]{HIM22}).
Hence, we focus on the proof of conclusion (1).

We first consider the case where $\mathcal{E}$ is pseudo-effective. 
In this case, by \cite[Proposition 2.4]{Mat22}, there exist singular Hermitian metrics $h_m$ on $\Sym^{[m]} \mathcal{E}$ 
such that the function 
$
f_m\coloneqq (1/m) \log | \tau^{ m} |_{h^\vee_m} 
$
satisfies that 
$
\deldel f_m \geq -(1/m) \omega. 
$
Here $\omega$ is a fixed K\"ahler form on $X$ and $\tau^{ m}$ is 
the section of $(\mathcal{E}^\vee)^{[m]}=(\mathcal{E}^{[m]})^{\vee}$ induced by $\tau$. 
Let $\pi \colon  \widetilde{X} \to X$ be a resolution of singularities of $X$. 
Since the pullback $\pi^{*} f_{m}$ satisfies that $\deldel f_m \geq -(1/m) \pi^{*}\omega$, 
the weak limit (after taking a subsequence) should be zero. 
If  $\tau$ has the zero point  $p \in X_{\mathcal{E}}\cap X_{\reg}$, 
the Lelong number of $f_m$ (and thus the weak limit) at $p$ is strictly greater than zero. 
This contradicts the fact that the weak limit is zero. 

We finally consider the case where $\mathcal{E}$ is almost nef. 
Taking general hypersurfaces passing through $p$, we obtain  a curve $C$ 
such that $p \in C \not\subseteq \mathbb{S}(\mathcal{E})$ and $C \subseteq X_{\reg} \cap X_{\mathcal{E}}$. 
The restriction $\mathcal{E}|_{C} $ is a nef vector bundle  by $C \not\subseteq \mathbb{S}(\mathcal{E})$ and $C \subseteq X_{\reg} \cap X_{\mathcal{E}}$, 
and  the section  $\tau|_{C} \in H^{0}(C,(\mathcal{E}|_{C})^{\vee})$ is non-zero. Hence, the section $\tau|_{C}$  should be nonvanishing on $C$. 
\end{proof}

The following lemma is an analog of \cite[Propositions 3.1, 3.2]{Mat22} originally formulated for pseudo-effective sheaves. 

\begin{lem}
\label{lem-p-curved}
Let $f \colon  X \to Y$ be a surjective morphism between normal projective varieties, 
and let $\mathcal{E}$ be a torsion-free sheaf on $Y$. 
If the reflexive pullback $f^{[*]}\mathcal{E}$ is positively curved,  
then $\mathcal{E}$ is positively curved. 
In particular, if the tangent sheaf $\mathcal{T}_X$ is positively curved, then so is $\mathcal{T}_{Y}$. 
\end{lem}

\begin{proof}
It is sufficient to construct a positively curved metric on $\mathcal{E}|_{Y_{\mathcal{E}} \cap Y_{\reg}}$. 
Hence,  by replacing $Y$ with $Y_{\mathcal{E}}\cap Y_{\reg}$, 
we may assume that $\mathcal{E}$ is a locally free and $Y$ is smooth. 
By taking the Stein factorization, 
we divide the proof into the case where  $f \colon  X \to Y$ is a fibration and 
the case where $f \colon  X \to Y$ is a surjective finite morphism. 

We first consider the case $f \colon  X \to Y$ is a fibration. 
Let $h$ be a  positively curved metric on $ f^{*}\mathcal{E}$.
Take a local section $\tau$ of  $\mathcal{E}^{\vee}$ and 
consider the norm $|f^{*}\tau|_{h^{\vee}}$ of the section $f^{*}\tau$  of $ f^{*}\mathcal{E}^\vee$, 
where $h^{\vee}$ is the dual metric on $\mathcal{E}^{\vee}$. 
Since the function $\log |f^{*}\tau|_{h^{\vee}}$  is psh,  
the norm $|f^{*}\tau|_{h^{\vee}}$ is constant along fibers. 
Since fibers are connected, 
the norm $|f^{*}\tau|_{h^{\vee}}$ can be written as $|f^{*}\tau|_{h^{\vee}}=f^{*} g_{\tau}$ 
for some psh function $g_{\tau}$. 
The functions $g_{\tau}$ determine the positively curved metric on $\mathcal{E}$. 
Indeed, since $h$ satisfies the parallelogram law, so is the function $g_{\tau}$. 
Hence  the functions $g_{\tau}$ determine the Hermitian metric on $\mathcal{E}$. 

We now consider the case $f \colon  X \to Y$ is a finite morphism. 
Take a subvariety $V \subseteq Y$ such that $f \colon  X \setminus f^{-1}(V) \to Y \setminus V$ is a finite \'etale cover. 
For  a local section $\tau$ of  $\mathcal{E}^{\vee}$, 
the function $g_{\tau} $ on $Y \setminus V$ defined by 
$$
g_{\tau} (p)\coloneqq \sum_{q \in f^{-1}(p)} |f^{*}\tau|_{h^{\vee}}(q)
$$
is psh. 
In the same way as above, these functions $g_{\tau} $ determine the positively curved metric on $\mathcal{E}|_{Y\setminus V}$. 
The function $g_{\tau}$ is bounded below around $V$. 
Hence, this metric can be extended to the metric on $\mathcal{E} $. 

For the latter conclusion, we consider the natural generically surjective sheaf morphism $\mathcal{T}_X \to f^*\mathcal{T}_Y$ over $Y_{\reg}$. 
Since $\mathcal{T}_X $ is positively curved, and so is $f^*\mathcal{T}_{Y_{\reg}}$. 
The positively curved metric on $f^*\mathcal{T}_{Y_{\reg}}$ descends to that on $\mathcal{T}_{Y_{\reg}}$, 
which  can be extended to the  positively curved metric on $\mathcal{T}_{Y}$ by $\codim Y_{\sing} \geq 2$. 
\end{proof}

At the end of this subsection, we prepare the two lemmas below. 
 
\begin{lem}\label{lem-bircontr-Q-abelian}
Let \(X\) be a finite $($not necessarily \'etale$)$ quotient of an abelian variety \(A\).
Then, any birational morphism \(X\to Y\) onto a normal projective variety \(Y\) is an isomorphism.
\end{lem}
\begin{proof}
Let us consider the following commutative diagram: 
\[
\xymatrix{
A\ar[r]\ar[d]&Y'\ar[d]\\
X\ar[r]&Y, 
}
\]
where  \(A\to X\) is a finite morphism from an abelian variety $A$, 
the morphism \(A\to Y'\) is the birational morphism, and  \(Y'\to Y\) is the finite morphism 
induced by the Stein factorization of the composite morphism \(A\to X\to Y\). 
Since \(A\) is an abelian variety,
the pullback of any ample divisor from \(Y'\) to \(A\) is still ample, 
which implies there is no curve contracted by the projection formula. 
Hence,  we obtain that $A$ is actually isomorphic to $Y'$. 
Since there is no curve contracted and \(X\to Y\) has connected fibers, 
the birational morphism \(X\to Y\) is also an isomorphism. 
\end{proof}

Let us recall functorial resolutions. 
A functorial resolution is a resolution of singularities of normal projective varieties such 
that it commutes with smooth morphisms and the exceptional locus is a simple normal crossing divisor 
(see \cite[Theorem 3.35, 3.45]{Kol07} and \cite[Section 4]{GKK10}).
If \(X\) is a normal projective surface, then the  minimal resolution of $X$ 
is a functorial resolution (cf.~\cite[Proposition 1.2]{Wah75}).

\begin{lem}[{\cite[Corollary 4.7]{GKK10}}]
\label{lem-functorial-resolution-tangent}
Let $X$ be a normal projective variety  and 
$\pi \colon  \widetilde{X} \to X$ be the functorial resolution with the exceptional divisor \(E\).
Then, there exist generically surjective sheaf morphisms
$$
\pi^{*} \mathcal{T}_{X} \to \mathcal{T}_{\widetilde{X}}(-\textup{log}\,E)
\text{ and }\pi^{*} \mathcal{T}_{X} \to \mathcal{T}_{\widetilde{X}}
 $$
that are isomorphisms over $X_{\reg}$.
\end{lem}
\begin{proof}
Since there is a generically surjective sheaf morphism \(\mathcal{T}_{\widetilde{X}}(-\text{log}\,E)\to \mathcal{T}_{\widetilde{X}}\), 
we are left to show the existence of the first morphism.
As  proved in \cite[Corollary 4.7]{GKK10}, the direct image \(\pi_*\mathcal{T}_{\widetilde{X}}(-\text{log}\,E)\) is reflexive, 
and hence \(\pi_*\mathcal{T}_{\widetilde{X}}(-\text{log}\,E)=\mathcal{T}_X\) holds.
In particular, via the trace map, we obtain  
$$ \pi^*\mathcal{T}_X=\pi^*\pi_*\mathcal{T}_{\widetilde{X}}(-\text{log}\,E)\to \mathcal{T}_{\widetilde{X}}(-\text{log}\,E). $$
By construction, this morphism is an isomorphism on $\pi^{-1}(X_{\text{reg}})$.
\end{proof}

\section{Flatness of reflexive sheaves on maximally quasi-\'etale varieties}\label{Sec-3}

In this section, we consider the structure of pseudo-effective or almost nef reflexive sheaves on projective klt varieties. 
For this purpose, we first review maximally quasi-\'etale covers and $\Q$-Chern classes.
A normal projective variety $X$ is said to be \textit{maximally quasi-\'etale} if the homomorphism between the \'etale fundamental groups 
$$
\widehat{\pi}_{1}(X_{\reg}) \rightarrow \widehat{\pi}_{1}(X)
$$
induced by 
$
i_{*}\colon \pi_{1}(X_{\reg}) \rightarrow \pi_{1}(X)
$
is an isomorphism, where $i \colon X_{\reg} \to X$ is the natural inclusion. 
By \cite[Theorem 1.14]{GKP16},  any projective klt variety admits a finite quasi-\'etale cover $\widehat{X} \rightarrow X$ 
such that $\widehat{X} $ is maximally quasi-\'etale. 
Note that any finite quasi-\'etale cover of $\widehat{X} $ is always \'etale by definition.

Let $X$ be a projective klt variety of dimension $n$, and 
let $\mathcal{E}$ and $\mathcal{F}$ be reflexive sheaves on $X$.
By \cite[Chapter 10]{K++}  and \cite[Theorem 3.13]{GKPT19b},
there exist symmetric $\Q$-multilinear forms on the N\'eron-Severi space with \(\mathbb{Q}\)-coefficients:
\begin{align*}
\widehat{c}_1(\mathcal{E})& \colon  \text{N}^{1}(X)_{\Q}^{n-1} \rightarrow \Q, &\quad 
&(\alpha_1, \ldots, \alpha_{n-1}) \mapsto \widehat{c}_1(\mathcal{E})\alpha_1\cdots\alpha_{n-1}, \\
\widehat{c}_1(\mathcal{E})\widehat{c}_1(\mathcal{F})&\colon  \text{N}^{1}(X)_{\Q}^{n-2} \rightarrow \Q, &\quad 
&(\alpha_1, \ldots, \alpha_{n-2}) \mapsto \widehat{c}_1(\mathcal{E})\widehat{c}_1(\mathcal{F}) \alpha_1\cdots\alpha_{n-2}, \\
\widehat{c}_2(\mathcal{E})& \colon  \text{N}^{1}(X)_{\Q}^{n-2} \rightarrow \Q, &\quad 
& (\alpha_1, \ldots, \alpha_{n-2}) \mapsto \widehat{c}_2(\mathcal{E})\alpha_1\cdots\alpha_{n-2},
\end{align*}
satisfying the following properties:

\begin{itemize}
\item [$($P1$)$.] In the case $n=2$, 
the surface $X$ admits a finite Galois (not necessarily quasi-\'etale) cover $\nu \colon  \widehat{X} \rightarrow X$  
such that $\nu^{[*]}\mathcal{E}$ is locally free and 
$$
(\deg \nu) \cdot (\widehat{c}_{1}(\mathcal{E}) \alpha) = c_{1}(\nu^{[*]}\mathcal{E}) \nu^{*}\alpha
$$
holds for any $\alpha \in \text{N}^{1}(X)_{\Q}$.

\item [$($P2$)$.]  In the case $n>2$, for an ample Cartier divisor $\mathcal{L}$ 
with a base point free linear system $\mathcal{B} \subseteq |\mathcal{L}|$ such that any hypersurface in $\mathcal{B}$ is connected, 
there exists a (non-empty) Zariski open set $\mathcal{B}_{0} \subseteq \mathcal{B} $ such that 
any member $V \in \mathcal{B}_{0}$ has klt singularities and satisfies that 
$\mathcal{E}|_{V}$ is reflexive and 
$$
\widehat{c}_1(\mathcal{E})\det(\mathcal{L})\alpha_2\cdots\alpha_{n-1}
=
\widehat{c}_1(\mathcal{E}|_{V})\alpha_2\cdots\alpha_{n-1}
$$
for any $\alpha_i \in \text{N}^{1}(X)_{\Q}$. 
\end{itemize}

The same properties hold for $\widehat{c}_1(\mathcal{E})\widehat{c}_1(\mathcal{F}) $ and $\widehat{c}_2(\mathcal{E})$. 
Note that we have 
$$\widehat{c}_1(\mathcal{E})\alpha_1\cdots\alpha_{n-1} = \det(\mathcal{E})\alpha_1\cdots\alpha_{n-1}$$
by (P1), (P2) and the projection formula (cf.\,\cite[Proposition 2.3 (c)]{Ful93}). 

\begin{prop}
\label{prop-almostnef-c2-locallyfree} 
Let  $X$ be a  projective klt variety of dimension $n$ and $H$ an ample Cartier divisor. 
Let $\mathcal{E}$ be a pseudo-effective or almost nef reflexive sheaf on $X$. 
If $\widehat{c}_1(\mathcal{E})H^{n-1}=0$, then 
$$
\widehat{c}_2(\mathcal{E}) H^{n-2}=\widehat{c}_1(\mathcal{E})^{2} H^{n-2}=0.
$$ 
In addition, if $X$ is maximally quasi-\'etale, then  $\mathcal{E}$ is a flat locally free sheaf.
\end{prop}

\begin{proof}

The first assertion has been proved when $X$ is a smooth surface by \cite[Corollary 2.12]{CH19} 
(cf.\,\cite[Theorem 1.2]{HIM22}, \cite[Theorem 1.8]{HP19}).
We reduce the higher dimensional case to this case as follows: 
We may assume that $X$ is a (possibly singular) surface by property (P2). 
Then, by property (P1), there exists a finite Galois cover $\nu \colon  \widehat{X} \rightarrow X$ 
such that $\nu^{[*]}\mathcal{E}$ is locally free and 
\begin{equation}
\label{eq-first-chern-qchern}
c_{1}(\nu^{[*]}\mathcal{E}) \nu^{*}H  =(\deg \nu) \cdot (\widehat{c}_{1}(\mathcal{E})H)  = 0 \text{ holds}. 
\end{equation}
Note that $\nu^{[*]}\mathcal{E}$ is pseudo-effective or almost nef by Lemma \ref{lem-elementary-lem-almostnef} (4). 
Take the minimal resolution $\pi \colon  \widetilde{X} \to \widehat{X}$. 
Then, by the functoriality of Chern classes of locally free sheaves, 
we obtain 
\begin{equation}
\label{eq-first-chern-minimalresol}
c_i(\pi^{*}\nu^{[*]}\mathcal{E}) = \pi^{*}c_i(\nu^{[*]}\mathcal{E})
\end{equation} 
for $i=1,2$. 
Since $\nu^{[*]}\mathcal{E}$ is a pseudo-effective or almost nef locally free sheaf,  
we obtain $c_1(\pi^{*}\nu^{[*]}\mathcal{E})=0$ by (\ref{eq-first-chern-qchern}) (cf.~\cite[Theorem 1.18]{KM98}).
By applying the results for smooth surfaces, we obtain 
$$c_1(\pi^{*}\nu^{[*]}\mathcal{E})^2 = c_2(\pi^{*}\nu^{[*]}\mathcal{E})=0.$$
Thus, the first conclusion follows from  (\ref{eq-first-chern-minimalresol}) and  property (P1).

We finally check the last assertion. 
The sheaf $\mathcal{E}$ is \(H^{n-1}\)-semistable by $\mu_{H}^{\min}(\mathcal{E}) = \mu_{H}(\mathcal{E}) =0$.
Then, we use the argument of \cite[Section 3.2 p.310]{LT18} to conclude. 
\end{proof}

\begin{cor}
\label{cor-c1-almostnef-tangent} 
Let $X$ be a projective klt variety of dimension $n$ with pseudo-effective or almost nef tangent sheaf. 
If $K_X  H^{n-1}\ge 0$ for some ample Cartier divisor $H$, 
then there exists a finite quasi-\'etale Galois cover $A \rightarrow X$ from an abelian variety $A$.
\end{cor}
\begin{proof}
By \cite[Theorem 1.14]{GKP16}, replacing $X$ with a finite quasi-\'etale cover, we may assume that $X$ is maximally quasi-\'etale.
Since  \(\mathcal{T}_X\) is generically nef, we have 
$$\widehat{c}_1(\mathcal{T}_{X})H^{n-1}=-K_X H^{n-1}=0.$$
Thus, by Proposition \ref{prop-almostnef-c2-locallyfree},  the tangent sheaf $\mathcal{T}_{X}$ is a flat locally free sheaf, 
and we can conclude that $X$ is smooth by the known case of the Zariski-Lipman conjecture in \cite[Theorem 6.1]{GKKP11}.
Then, by Yau's theorem, we see that $X$ is an abelian variety after taking a further \'etale cover.
\end{proof}

To generalize the Fujita decomposition (Proposition \ref{thm-fujita-decomposition-klt}), 
we prepare the following proposition: 

\begin{prop}
\label{prop-shm-hermflat} 
Let $\mathcal{E}$ be a reflexive sheaf on a maximally quasi-\'etale projective klt variety $X$.
Assume that $\mathcal{E}$ has a positively curved singular Hermitian metric $h$.
If $\widehat{c}_1(\mathcal{E}) H^{n-1} =0$ for some ample Cartier divisor $H$, then 
$\mathcal{E}$ is locally free and $h$ is a smooth Hermitian flat metric on $X$. 
\end{prop}
\begin{proof}
The sheaf $\mathcal{E}$ is a flat locally free sheaf by Proposition \ref{prop-almostnef-c2-locallyfree}. 
Take a resolution $\pi \colon  \widetilde{X} \to X$. Since $\mathcal{E}$ is locally free, 
the pullback $\pi^{*}h$ is a positively curved  metric on $\pi^{*}\mathcal{E}$. 
By \cite[Lemma 3.5]{HIM22}, the pullback $\pi^{*} h$ is a smooth Hermitian flat metric, and thus so is $h $. 
\end{proof}

\begin{cor}
\label{cor-shm-split} 
Let $X$ be a maximally quasi-\'etale projective klt variety and let
\begin{equation}
\label{eq-exact-refexive-1} 
0 \longrightarrow \mathcal{S}  \longrightarrow \mathcal{E}\longrightarrow  \mathcal{Q} \longrightarrow 0
\end{equation}
be an exact sequence of  reflexive sheaves.
If $\mathcal{E}$ is positively curved and $\widehat{c}_1(\mathcal{Q})H^{n-1} =0$ for some ample Cartier divisor $H$, 
then the above sequence splits $($i.e.,\,$\mathcal{E} \cong \mathcal{S}\oplus \mathcal{Q}$$)$.
\end{cor}
\begin{proof}
Take a big open set $U \subseteq X$ so that $U$ is smooth and the reflexive sheaves in (\ref{eq-exact-refexive-1}) are  locally free over $U$.
Since $\mathcal{Q}$ admits a quotient metric $h_{\mathcal{Q}}$ that is positively curved, which is Hermitian flat on $U$ by Proposition \ref{prop-shm-hermflat}.
By \cite[Theorem 3.8]{HIM22},  the exact sequence (\ref{eq-exact-refexive-1})  splits on $U$:
$$
\mathcal{E}|_{U} \cong \mathcal{S}|_{U} \oplus \mathcal{Q} |_{U}.
$$
Taking the reflexive hull and noting that \(\text{codim}(X\backslash U)\ge 2\) holds, we obtain the desired conclusion. 
\end{proof}

We now prove the Fujita decomposition for reflexive sheaves on maximally quasi-\'etale projective klt varieties. $($cf.\,\cite{Iwa22}, \cite{LS22}$)$. 
\begin{thm}
\label{thm-fujita-decomposition-klt}
Let $\mathcal{E}$ be a reflexive sheaf on a maximally quasi-\'etale projective klt variety $X$.
If $\mathcal{E}$ is pseudo-effective or almost nef, 
then there uniquely exists the exact sequence  
$$
0 \longrightarrow  \mathcal{S} \longrightarrow \mathcal{E} \longrightarrow \mathcal{Q} \longrightarrow 0
$$
with torsion-free sheaves $ \mathcal{S}$ and $\mathcal{Q}$ satisfying the following properties$:$ 
\begin{enumerate}
\item[$(1)$] The reflexive hull  $\mathcal{Q}^{\vee\vee}$ is a flat locally free sheaf.
\item[$(2)$] $ \mathcal{S}$ is a pseudo-effective $($resp.\,almost nef$)$ generically ample reflexive sheaf. 
\end{enumerate}
Furthermore, if $\mathcal{E}$ is positively curved, then we have  
$
\mathcal{E}  \cong  \mathcal{S} \oplus \mathcal{Q}^{\vee\vee}.
$
\end{thm}
\begin{proof}
Take an ample Cartier divisor $H$. 
If $\mu_{H}^{\text{min}}(\mathcal{E}) >0$, then we take $ \mathcal{S}\coloneqq  \mathcal{E} $ and $\mathcal{Q} \coloneqq 0$.
Thus, we may assume that $\mu_{H}^{\text{min}}(\mathcal{E}) =0$ (cf.~Proposition \ref{pro-basic-implication}).
Consider the Harder-Narasimhan filtration of $\mathcal{E}$ with respect to $H^{n-1}$:
$$
0 = \mathcal{E}_0 \subsetneq \mathcal{E}_1 \subsetneq \cdots \subsetneq \mathcal{E}_l =\mathcal{E}.
$$
Set $ \mathcal{S}\coloneqq  \mathcal{E}_{l-1} $ and $\mathcal{Q} \coloneqq  \mathcal{E}/\mathcal{E}_{l-1}$.
Since $\mathcal{Q}$ is pseudo-effective (or almost nef)  and $\mu_{H}(\mathcal{Q}) = \mu_{H}^{\text{min}}(\mathcal{E})=0 $, by Proposition \ref{prop-almostnef-c2-locallyfree}, the reflexive hull $\mathcal{Q}^{\vee\vee}$ is a flat locally free sheaf.
Thus,  the sheaf $\mathcal{S}$ is pseudo-effective $($or almost nef$)$ by noting that there is a generically surjective morphism
$$
\wedge^{\text{rank}\mathcal{Q}+1}\mathcal{E}\otimes \det(\mathcal{Q}^{\vee}) \to  \mathcal{S}.
$$
The quotient sheaf $\mathcal{E}/ \mathcal{S}$ is torsion-free, and thus $ \mathcal{S}$ is saturated in $\mathcal{E}$. 
In particular, the sheaf $ \mathcal{S}$ is reflexive.

The minimum slope satisfies that $\mu^{\min}_{H}(\mathcal{S}) >0$ by $\mu_{H}^{\text{min}} (\mathcal{E}_{i}/\mathcal{E}_{i-1}) >0$ for any $1 \le i \le l-1$. 
We now prove that $\mathcal{S}$ is generically ample. 
For this purpose, we suppose that $\mu^{\min}_{H_1 \cdots H_{n-1}} (\mathcal{E}) =0$ for some ample Cartier divisors $H_{1}, \ldots, H_{n-1}$. 
Thus, by the above argument, we can take a quotient sheaf $\mathcal{S} \to \mathcal{R}$ such that 
$\mathcal{R}^{\vee\vee}$ is a flat locally free sheaf, which contradicts  $\mu^{\min}_{H}(\mathcal{S}) >0$.

The uniqueness follows from the following fact: any sheaf morphism $\mathcal{G} \to \mathcal{H}$ is the zero map if $\mu_{H}^{\min}(\mathcal{G}) > \mu_{H}^{\max}(\mathcal{H})$. 
The last assertion for positively curved sheaves follows from Corollary \ref{cor-shm-split}.
\end{proof}

\begin{rem}
\label{rem-maximal-flat-locally-free}
Under the setting in Theorem \ref{thm-fujita-decomposition-klt}, the maximal destabilizing subsheaf of $\mathcal{E}^{\vee}$ is independent of the choice of ample Cartier divisors $H_1,\ldots, H_{n-1}$.
Moreover, the maximal destabilizing subsheaf of $\mathcal{E}^{\vee}$ is the maximal flat locally free sheaf contained in $\mathcal{E}^{\vee}$.  
Indeed, for the maximal destabilizing subsheaf $\mathcal{E}^{\vee}_{\max}$ of $\mathcal{E}^{\vee}$ with respect to 
$\alpha \coloneqq H_1 \cdots H_{n-1}$ for any ample Cartier divisors $H_1,\ldots, H_{n-1}$, 
we consider the following morphism$:$
$$\gamma \colon  \mathcal{E}^{\vee}_{\max} \longrightarrow \mathcal{E}^{\vee} \longrightarrow \mathcal{S}^{\vee}.$$
By $\mu_{\alpha}^{\text{min}}(\mathcal{E}^{\vee}_{\max}) > \mu_{\alpha}^{\text{max}}(\mathcal{S}^{\vee}) $, the morphism $\gamma$ is the zero map. 
Thus, by the universality of kernel, we obtain an injection $\mathcal{E}^{\vee}_{\max} \to \mathcal{Q}^{\vee}$, which implies that  $\mathcal{E}^{\vee}_{\max} = \mathcal{Q}^{\vee}$ by the semi-stability of \(\mathcal{Q}^{\vee}\).
\end{rem}

\begin{prop}
\label{prop-genericallyample-etale}
Let $\mathcal{E}$ be a reflexive sheaf of a projective klt variety $X$.
Assume that $\mathcal{E}$ is pseudo-effective or almost nef.
Then, we have$:$
\begin{enumerate}
\item[$(1)$] The maximal destabilizing subsheaf of $\mathcal{E}^{\vee}$ is independent of the choice of ample Cartier divisors $H_1, \ldots, H_{n-1}$.
\item[$(2)$] Let $\nu \colon  \widehat{X} \to X$ be a finite quasi-\'etale Galois cover such that $\widehat{X}$ is maximally quasi-\'etale. 
Then,  the maximal destabilizing sheaf of $(\nu^{[*]}\mathcal{E} )^{\vee}$  coincides with the reflexive pullback of that of $\mathcal{E}^{\vee}$. 
In particular, the sheaf $\mathcal{E}$ is generically ample if and only if $\nu^{[*]}\mathcal{E}$ is generically ample.
\end{enumerate}
 \end{prop}
 
 \begin{proof}
 Conclusion (1) follows from (2) and Remark \ref{rem-maximal-flat-locally-free}.
 For the proof of conclusion (2), we fix ample Cartier divisors $H_1,\ldots, H_{n-1}$ on $X$. Set  $\alpha \coloneqq  H_1\cdots H_{n-1}$ and $\widehat{\mathcal{E}} \coloneqq  \nu^{[*]}\mathcal{E}$.
 Take the  maximal destabilizing sheaf $\mathcal{G}$ (resp.\,$\widehat{\mathcal{G}}$) of $\mathcal{E}^{\vee}$ (resp.\,$\widehat{\mathcal{E}}^{\vee}$) with respect to $\alpha$ (resp.\,$\nu^{*}\alpha$). 
Then, we can easily see that $\nu^{[*]}\mathcal{G} \subseteq \widehat{\mathcal{G}}$.
On the other hand, 
by \cite[Proposition 2.16]{GKPT19b}, there is a saturated subsheaf $\mathcal{F} \subseteq \mathcal{E}^{\vee}$ with $\widehat{\mathcal{G}}= \nu^{[*]}\mathcal{F}$. 
Thus, we obtain $\mathcal{F} \subseteq \mathcal{G}$, and hence $\widehat{\mathcal{G}} \subseteq \nu^{[*]}\mathcal{G}$.
Under this situation, the sheaf $\mathcal{E}$ is generically ample if and only if $\mathcal{G}=0$, and thus the last assertion of (2) is obtained. 
\end{proof}

\section{Projective klt varieties with positively curved tangent sheaf}\label{Sec-4}

This section is devoted to the proof of Theorem \ref{thm-positively-curved}.
We first prove the following lemma for the reader's convenience though it is known to experts.

\begin{lem}
\label{lem-qfact} 
Let $\phi \colon  X \to Y$ be an equidimensional fibration between normal projective varieties. 
Then, we have$:$
\begin{itemize}
\item[$(1)$] If $X$ is $\mathbb{Q}$-factorial, then so is $Y$. 
\item[$(2)$] If $X$ has klt singularities and $K_Y$ is $\mathbb{Q}$-Cartier, 
then $Y$ has klt singularities. 
\end{itemize}
\end{lem}
\begin{proof}
We reduce the proofs to the case where $\phi \colon  X \to Y$ is a finite morphism. 
Taking general hypersurfaces $\{H_i\}_{i=1}^m$ on $X$, 
we consider the variety $Z\coloneqq X \cap H_1 \cap \cdots \cap H_m$, where $m\coloneqq \dim X-\dim Y$. 
Then, the induced morphism $\phi|_Z \colon  Z \to Y$ is a finite (surjective) morphism by the equidimensionality of $\phi \colon  X \to Y$. 

\smallskip
(1). 
Let $D$ be a Weil divisor on $Y$. 
The pullbacks $\phi^* D$ and $(\phi|_Z) ^* D$ can be defined as a Weil divisor by equidimensionality 
(for example, see \cite[Definition 3.5]{Dru18}). 
The Weil divisor $\phi^* D$ is $\mathbb{Q}$-Cartier by assumption, 
so is $(\phi|_Z) ^* D=(\phi^* D)|_Z$. 
Since $\phi|_Z \colon  Z \to Y$ is a finite morphism, 
the Weil divisor $D$ is also $\mathbb{Q}$-Cartier by \cite[Lemma 5.16]{KM98}.

\smallskip
(2). 
The conclusion is local in $Y$. 
Taking an index one cover $Y' \to Y$ of $Y$ and replacing $X$ (resp.\,$Z$) with the fiber product $X \times_{Y} Y'$ (resp.\,$Z \times_{Y} Y'$),  
we may assume that $K_Y$ is Cartier. 
By construction, the variety $Z$ has klt singularities, and thus $Z$ has rational singularities. 
Since the restriction $\phi|_{Z} \colon  Z \to Y$ is a finite morphism, the base variety $Y$  also has rational singularities by \cite[Proposition 5.13]{KM98}.  
Since $K_Y$ is Cartier, the variety $Y$ has canonical singularities by \cite[Corollary 5.24]{KM98}. 
\end{proof}

We begin the proof of Theorem \ref{thm-positively-curved}. 

\begin{proof}[Proof of Theorem \ref{thm-positively-curved}]

Let $X$ be a projective klt variety with positively curved tangent sheaf. 
We first show that, under the additional assumption that $X$  is $\mathbb{Q}$-factorial and maximally quasi-\'etale,  
there exists an equidimensional MRC fibration $X \to Y$ such that $Y$ is a quasi-\'etale quotient of an abelian variety, 
from which the general case will be deduced at the end of the proof. 
To this end, we prove Proposition \ref{prop-wreg} under a somewhat technical situation explained below.

Take an MRC fibration $\phi \colon  X \dashrightarrow Y$ onto a normal projective variety $Y$ and 
a resolution $\pi \colon  \widetilde X \to X$ of the indeterminacy locus of $\phi$ and singularities of $X$. 
We have the following commutative diagram$:$
\begin{equation*}
\xymatrix@C=40pt@R=30pt{
\widetilde X  \ar[rr]^\pi \ar[rd]_{\varphi\ \ \ }  &  & X \ar@{-->}[ld]^{\phi \ \ \ }\\ 
&   Y. &  \\   
}
\end{equation*}
Hereafter, we assume that the canonical divisor $K_{Y}$ is a pseudo-effective $\mathbb{Q}$-Cartier divisor, \(Y\) is klt, and 
$\codim \pi(\varphi^{-1}(Y_{\sing})) \geq 2$ holds. 
This assumption, which looks technical, is naturally satisfied in our applications (see Remark \ref{rem-wreg}). 

\begin{rem}\label{rem-wreg}
(1) The above assumption is automatically satisfied if $Y$ is a smooth projective variety. 
Indeed,  the smooth base variety $Y$ of MRC fibrations is non-uniruled by \cite{GHS03}. 
Hence, the canonical divisor $K_Y$ is pseudo-effective by \cite{BDPP13}. 
\\
(2) The assumption of $\codim \pi(\varphi^{-1}(Y_{\sing})) \geq 2$  is satisfied
if $\phi \colon  X \dashrightarrow Y$ is an everywhere-defined  equidimensional fibration. 
\end{rem}

The cotangent sheaf $\Omega_{Y}$ is  locally free on $Y_{\reg}$.  
Hence, we have the natural injective sheaf morphism ${\varphi}^{*} \Omega_Y \to \Omega_{\widetilde X}$ 
on $\widetilde X \setminus \varphi^{-1}(Y_{\sing})$. 
By our assumption, the push-forward $\pi_*$ induces 
the injective sheaf morphism 
$\pi_*{\varphi}^{*} \Omega_Y \to \pi_*\Omega_{\widetilde X} $ on a big Zariski open set $X \setminus \pi(\varphi^{-1}(Y_{\sing}))$ of $X$. 
By considering the dual sheaves, 
we obtain the exact sequence:  
\begin{align}\label{eq-1}
0 \longrightarrow \mathcal{S}\coloneqq  \Ker r \longrightarrow 
\mathcal{T}_X \xrightarrow{\quad r \quad} \mathcal{Q}\coloneqq  (\pi_*{\varphi}^{*} \Omega_Y)^{\vee} \text{ on } X. 
\end{align}
Note that the sheaf morphism $\mathcal{T}_X \to \mathcal{Q}$ is generically surjective by construction. 
Then, we have$:$

\begin{prop}\label{prop-wreg}
In the above situation, we have$:$
\begin{itemize}
\item[$(1)$] $\mathcal{Q}$ is a Hermitian flat locally free sheaf on $X$. 
\item[$(2)$] The exact sequence \eqref{eq-1} splits on $X$ $($i.e.,\,$\mathcal{T}_X \cong \mathcal{S} \oplus \mathcal{Q}$ holds on $X$$).$
\item[$(3)$] $\mathcal{S} \subseteq \mathcal{T}_X $ is an algebraically integrable and weakly regular foliation. 
\item[$(4)$] $\det \mathcal{Q} \sim_\mathbb{Q} 0$.  
\end{itemize}
\end{prop}
\begin{proof}
(1). 
We first apply Proposition \ref{prop-almostnef-c2-locallyfree} to show that $\mathcal{Q}$ is locally free. 
It is sufficient for this purpose to check that $c_{1}(\mathcal{Q})=0$ since positively curved sheaves are always pseudo-effective. 
Since the sheaf morphism $\mathcal{T}_X \to \mathcal{Q}$ is generically surjective, the sheaf $\mathcal{Q}$ is positively curved. 
In particular, the first Chern class $c_{1}(\mathcal{Q})$ is pseudo-effective.  
Meanwhile, since the canonical divisor $K_Y$ is a pseudo-effective $\mathbb{Q}$-Cartier divisor by assumption, 
we can show that $-c_{1}(\mathcal{Q})$ is pseudo-effective. 
Indeed, let $\pi_* \varphi^*  K_Y$ denote the push-forward of the $\mathbb{Q}$-Cartier divisor $\varphi^*  K_Y$ as a Weil divisor, and 
take a positive integer $m \in \mathbb{Z}_{+}$ such that both $m (\pi_* \varphi^*  K_Y)$ and $  \varphi^*  (mK_Y)$ are Cartier 
by the $\mathbb{Q}$-factoriality of $X$. 
Note that the associated invertible sheaf $\mathcal{O}_X (m \pi_* \varphi^*  K_Y) $ might be different from 
the direct image sheaf $\pi_* \mathcal{O}_{\widetilde X}(  \varphi^* m K_Y)$, 
but it coincides with the reflexive hull of $\pi_* \mathcal{O}_{\widetilde X}(  \varphi^*  mK_Y)$. 
Set 
$$X_{1}\coloneqq X \setminus (X_{\phi} \cup X_{\sing} \cup \pi(\varphi^{-1}(Y_{\sing}))),$$ 
where $X_{\phi}$ be the minimal Zariski closed subset in $X$ 
where the restriction $\phi|_{X \setminus X_{\phi}}\colon  X \setminus X_{\phi} \to Y $ is an everywhere-defined morphism. 
Note that $X_{1}$ is a big Zariski open set in $X$ by the assumption $\codim \pi(\varphi^{-1}(Y_{\sing})) \geq 2$. 
Since $$\mathcal{Q}=\pi_* \varphi^* \mathcal{T}_Y=(\phi|_{X_1})^*\mathcal{T}_Y$$ holds on $X_1$, 
we have $\det \mathcal{Q}=\mathcal{O}_{X_{1}}(-\pi_* \varphi^*K_{Y})$ on $X_{1}$. 
This indicates that 
$$
(\det \mathcal{Q})^{[\otimes -m]} = 
\mathcal{O}_X (m \pi_* \varphi^*  K_Y) = \big( \pi_* \mathcal{O}_{\widetilde X}(\varphi^* m   K_Y) \big)^{\vee\vee}
$$ 
by $\codim (X \setminus X_1) \geq 2$. 
Since $\varphi^*  K_Y$ is pseudo-effective as a $\mathbb{Q}$-Cartier divisor, 
so is the right-hand side, which indicates that the first Chern class $-c_1(\mathcal{Q})$ is also pseudo-effective. 
Thus, we can conclude  that $c_{1}(\mathcal{Q})=0$. 
Since $X$ is maximally quasi-\'etale by assumption, 
the sheaf $\mathcal{Q}$ is a flat locally free sheaf on $X$ by Proposition \ref{prop-almostnef-c2-locallyfree}. 
The metric $h_{q}$ on $\mathcal{Q}$ induced by a positively curved singular Hermitian metric on $\mathcal{T}_{X}$ is also positively curved. 
By Proposition \ref{prop-shm-hermflat}, we can conclude that $(\mathcal{Q}, h_{q})$ is Hermitian flat. 
\smallskip

(2). 
Lemma \ref{lem-nonvanish} shows that the sequence \eqref{eq-1} is an exact sequence as a bundle morphism on $X_\reg$  
by noting that $\mathcal{T}_{X}$ is locally free on $X_{\reg}$. 
Hence, by applying Corollary \ref{cor-shm-split}, we obtain the splitting $\mathcal{T}_X \cong \mathcal{S} \oplus \mathcal{Q}$.  
\smallskip

(3). 
The subsheaf $\mathcal{S} \subseteq \mathcal{T}_X$ is a saturated subsheaf since $\mathcal{Q}$ is torsion-free. 
By construction, the restriction $\mathcal{S}|_{X_1}$  coincides with the foliation induced by the fibration $\phi|_{X_{1}}\colon  X_1 \to Y$. 
Hence, the subsheaf $\mathcal{S}$ is an algebraically integrable foliation on $X$. 
By (2), the exact sequence \eqref{eq-1} splits on $X$. 
Hence, the foliation $\mathcal{S}$ is a weakly regular foliation.

\smallskip

(4). 
The difference $\varphi^*  (mK_Y) - \pi^* (m \pi_* \varphi^*  K_Y)$ is a $\pi$-exceptional divisor. 
The proof of (1) shows that $\pi_* \varphi^*  K_Y$ is numerically trivial by $c_{1}(\mathcal{Q})=0$. 
Hence, the divisor $\varphi^*  (mK_Y)$ can be written as the sum of a numerically trivial divisor and 
a (not necessarily effective) $\pi$-exceptional divisor, which implies that the numerical dimension \(\nu(\varphi^*(mK_Y))\) is no more than zero 
(cf.~\cite[Chapter V, Definition 2.5]{Nak04}).
On the other hand, since \(mK_Y\) is pseudo-effective, we have \(\nu(\varphi^*(mK_Y))\ge 0\) (cf.~\cite[Chapter V, Corollary 1.4]{Nak04}).
In particular, we obtain \(\nu(K_Y)=0\), and then it follows from \cite[Chapter V, Corollary 4.9]{Nak04} that  \(K_Y\) is abundant 
(i.e., there exists an effective divisor $D$ such that $K_Y \sim_\mathbb{Q} D$).
Since $\varphi^*D$ is $\pi$-exceptional, 
we can conclude that $\det \mathcal{Q} \sim_\mathbb{Q} \pi_{*} \varphi^*K_{Y} \sim_\mathbb{Q} 0$. 
\end{proof}

By Proposition \ref{prop-wreg} (3) and \cite[Theorem 17]{DGP20}, 
we obtain an equidimensional fibration $\phi \colon  X \to Y$ onto a normal projective variety $Y$, 
for which we use the same notation as the MRC fibration we took at the beginning of the proof.

We now prove that $Y$  is a quasi-\'etale quotient of an abelian variety. 
The canonical divisor $K_{Y}$ is $\mathbb{Q}$-Cartier and $Y$ has klt singularities by Lemma \ref{lem-qfact}. 
The fibration $\phi \colon  X \to Y$ is still an MRC fibration of $X$. 
Indeed, by construction,  the fibration $\phi \colon  X \to Y$ induced by $\mathcal{S}$ is birational 
to the MRC fibration we took at the beginning, which indicates that a general fiber of $\phi \colon  X \to Y$ is rationally connected. 
Furthermore, the variety $Y$ is non-uniruled, which indicates that $K_{Y}$ is pseudo-effective. 
The above argument shows that $\phi \colon  X \to Y$ satisfies the assumption of Proposition \ref{prop-wreg} by Remark \ref{rem-wreg} (2). 

Let us consider the exact sequence \eqref{eq-1} for this equidimensional fibration $\phi \colon  X \to Y$. 
Then, the sheaf $\mathcal{Q}$ is Hermitian flat by Proposition \ref{prop-wreg} (1). 
Since $\mathcal{Q}=\phi^* \mathcal{T}_Y$ holds on $\phi^{-1}(Y_{\reg})$, 
the tangent sheaf $\mathcal{T}_{Y}$ on $Y_{\reg}$ admits a Hermitian flat metric, 
which can be extended on $\mathcal{T}_{Y}$ by $\codim Y_{\sing} \geq 2$. 
In particular, we can conclude that $\mathcal{T}_Y$ is positively curved and $c_{1}(\mathcal{T}_{Y})=0$ holds. 
Hence, the base variety $Y$ is a quasi-\'etale quotient of an abelian variety by \cite[Theorem 1.2]{Gac20} 
(see also \cite[Proposition 3.3]{Mat22}). 

\smallskip

We finally complete the proof of Theorem \ref{thm-positively-curved}. 
Let $X$ be a projective klt variety with positively curved tangent sheaf. 
Take a maximally quasi-\'etale cover $\nu \colon  X' \to X$  of $X$ by \cite{GKP16} and 
a $\mathbb{Q}$-factorization $ X'_{\qt} \to X'$ that is a small birational morphism by \cite[Corollary 1.37]{Kol13}. 
We first check that $X'_{\qt} $ satisfies the assumption of the first half 
(i.e.,\,$X'_{\qt} $ is $\mathbb{Q}$-factorial and maximally quasi-\'etale, and $\mathcal{T}_{X'_{\qt} }$ is positively curved). 
The variety $X'_{\qt}$ is  maximally quasi-\'etale by \cite[Theorem 2.12]{MW21}. 
The tangent sheaves $\mathcal{T}_{X'}$ and $\mathcal{T}_{X'_{\qt}}$ are also positively curved. 
Indeed, a positively curved singular Hermitian metric $h$ on $\mathcal{T}_{X}$ induces 
the metric $\nu^{*}h$ on $\nu^{*}(\mathcal{T}_{X_{\reg}})=\mathcal{T}_{X'} |_{\nu^{-1}(X_{\reg})}$, 
which can be extended to the positively curved metric on $\mathcal{T}_{X'}$ by noting that $\codim \nu^{-1}(X_{\sing}) \geq 2$. 
The metric $\nu^{*}h$ induces  the positively curved metric on a big Zariski open set of $X'_{\qt}$
since $ X'_{\qt} \to X$  is a small birational morphism, and thus it can also be extended to a positively curved metric on \(\mathcal{T}_{X'_{\qt}}\).

By applying the results of the first half, 
we obtain an equidimensional MRC fibration $X'_{\qt} \to Y$ 
and a finite quasi-\'etale cover $\tau \colon A \to Y$  with an abelian variety $A$. 
Let $Z$ be the normalization of the fiber product $A \times_{Y} X'_{\qt}$ 
(noting that $A \times_{Y} X'_{\qt}$ is irreducible since $X'_{\qt} \to Y$ is equidimensional) and 
let $Z \to W \to X'$ be the Stein factorization of $Z \to X'$ (see the diagram on the left below). 

We will show that $W $ admits a locally constant MRC fibration onto an abelian variety. 
Since $X'_{\qt} \to Y$ is an equidimensional fibration and 
$A \to Y$ is a finite quasi-\'etale cover, 
the induced cover  $\mu \colon Z \to X'_{\qt}$ (and thus $W \to X'$) is also a finite quasi-\'etale cover. 
The cover $W \to X'$ is actually a finite \'etale cover since $X'$ is maximally quasi-\'etale. 
Since $Z \to W$ is birational by construction,  
there exists a dominant rational map $W \dashrightarrow A$, 
which actually determines the everywhere-defined fibration by \cite[Corollary 1.7]{HM07}, 
by noting that $W$ has klt singularities and $A$ contains no rational curves 
(see the diagram on the right below). 
$$
\xymatrix{
&W \ar@/^18pt/[dr] & \\
Z \ar[r]^{\mu}\ar[d]_{} \ar@/^18pt/[ur]& X'_{\qt}\ar[d]_{\phi}  \ar[r]& X'\\
A \ar[r]_{\tau} &Y, &
}
\qquad  \qquad  
\xymatrix{
&  & \\
\widetilde  W   \ar[dr]_\varphi \ar[r]^{\alpha} & W \ar[d]^{\psi}  \ar[r]& X'   \\\
 & A.&
}
$$

The tangent sheaf $\mathcal{T}_{W}$ is positively curved since so is $\mathcal{T}_{X'}$ (cf.~Lemma \ref{lem-p-curved}). 
Hence, by  Proposition \ref{prop-wreg} (2), 
we obtain the splitting $\mathcal{T}_{W}\cong  \mathcal{S} \oplus \psi ^{*} \mathcal{T}_{A}$. 
Strictly speaking, in Proposition \ref{prop-wreg}, 
we  assumed that $X$ (which is $W$ in the present situation) is $\mathbb{Q}$-factorial and maximally quasi-\'etale, 
but this assumption was used only to show that $\mathcal{Q}$ is locally free. 
Since  $\mathcal{Q}=\phi^{*} \mathcal{T}_{Y}$ is locally free in the present situation, 
we can obtain the desired splitting by the same argument as in Proposition \ref{prop-wreg} without using this assumption (cf.~Corollary \ref{cor-shm-split}). 

Let $\alpha \colon  \widetilde  W \to W$ be a functorial resolution of $W$. 
Then, by \cite[Lemma 5.10]{Dru18}, 
we obtain the splitting 
$$\mathcal{T}_{\widetilde  W}= \mathcal{T}_{\widetilde  W/A} \oplus (\psi\circ \alpha )^{*} \mathcal{T}_{A}.$$ 
Thus, the fibration $\varphi \colon \widetilde{W} \to A$ is a fiber bundle 
by Ehresmann's theorem (see also \cite[Lemma 3.19]{Hor07}), 
and hence a locally constant fibration by \cite[Lemma 4.1]{Mul}.
We now check that $W \to A$ is a locally constant fibration. 
Lemma \ref{lem-flat-reduced} shows that $W \to A$ is an equidimensional fibration, and thus flat by the miracle flatness. 
Let $L$ be an ample Cartier divisor on $W$. 
By \cite[Lemma 2.4]{MW21}, there is a $\mathbb{Q}$-line bundle $L_{A}$ on $A$
such that 
$$\varphi_{*} (m (\alpha^{*}L- (\psi\circ \alpha )^{*}L_{A}))$$ is a flat vector bundle 
for a divisible integer $m \in \mathbb{Z}_{+}$. 
Then, the direct image sheaf $\psi_{*} (m (L- \psi^{*}L_{A}))$ is flat and $(L- \psi^{*}L_{A})$ is $\psi$-ample. 
Thus, by \cite[Proposition 2.5]{MW21}, we conclude that $W \to A$ is a  locally constant fibration. 
\end{proof}

\section{Projective klt varieties with  almost nef tangent sheaf}\label{Sec-5}

This section is devoted to the proof of Theorem \ref{thm-almost-nef}. 
A key point of the proof is that the almost nefness of tangent sheaves is preserved by taking functorial resolutions,  
which enables us to deduce the desired conclusion from a structure theorem in the smooth case. 

\begin{proof}[Proof of Theorem \ref{thm-almost-nef}]

Let $X$ be a projective klt variety with almost nef tangent sheaf. 
Take a functorial resolution $\pi \colon  \widetilde{X} \to X$.  
Since there exists a generically surjective morphism $\pi^{*}\mathcal{T}_X \to \mathcal{T}_{\widetilde{X}}$ by Lemma \ref{lem-functorial-resolution-tangent}, 
the tangent bundle $\mathcal{T}_{\widetilde{X}}$ is almost nef by Lemma \ref{lem-elementary-lem-almostnef} (2), (4). 
Therefore, by applying \cite{Iwa22} (cf.\,\cite{HIM22}), we obtain a fibration $\varphi \colon  \widetilde{X} \to Y$ 
satisfying the following properties:
\begin{enumerate}
\item[$(1)$] $\varphi \colon  \widetilde{X} \to Y$ is a smooth MRC fibration onto a smooth projective variety $Y$. 
In particular, the standard sequence of tangent bundles is exact: 
$$
0 \to \mathcal{T}_{\widetilde {X}/Y} \to \mathcal{T}_{\widetilde{X}} \to \varphi^{*} \mathcal{T}_{Y} \to 0.
$$ 
\item[$(2)$] The base variety $Y$ is an \'etale quotient of an abelian variety. 
\end{enumerate}
For the reader's convenience, following the strategy in Section \ref{Sec-4}, we briefly explain how to obtain the fibration $\varphi \colon  \widetilde{X} \to Y$. 
Let $\widetilde{X} \dashrightarrow Y'$ be a (not necessarily holomorphic) MRC fibration to a smooth projective variety $Y'$. 
Let us consider the exact sequence \eqref{eq-1} by noting that $\mathcal{T}_{\widetilde{X}}$ is locally free and almost nef. 
By applying Lemma \ref{lem-nonvanish} for the almost nef \(\mathcal{T}_{\widetilde{X}}\), 
we obtain an exact sequence of vector bundles 
$$0\to\mathcal{S}\to \mathcal{T}_{\widetilde{X}}\to\mathcal{Q}\to 0.$$ 
The subbundle $\mathcal{S} \subseteq {T}_{\widetilde{X}}$ is a foliation whose general leaf is rationally connected 
by the same argument as in Proposition \ref{prop-wreg} (3). 
Thus, this foliation induces the smooth MRC fibration $\varphi \colon \widetilde{X} \to Y$ onto a smooth projective variety $Y$ (possibly different from \(Y'\))  by \cite[Theorem 3.17]{Hor07}. 
Meanwhile, the tangent bundle $\mathcal{T}_{Y}$ is almost nef thanks to the surjection $\mathcal{T}_{\widetilde{X}}\to\varphi^{*}\mathcal{T}_{Y}$ 
and $K_{Y}$ is pseudo-effective by \cite{BDPP13} since the smooth variety $Y$ is non-uniruled. 
Hence, the tangent bundle $\mathcal{T}_{Y}$ is  numerically flat  by Proposition \ref{prop-almostnef-c2-locallyfree},  
which indicates that $Y$ is an \'etale quotient of an abelian variety. 

Let us return to the proof of Theorem \ref{thm-almost-nef}. 
Since  $Y$ has no rational curves, the dominant rational map $X \dashrightarrow Y$  induced  by $\varphi \colon  \widetilde{X} \to Y$ 
actually determines an everywhere-defined fibration by \cite[Corollary 1.6]{HM07} and the rigidity lemma, with the following commutative diagram$:$ 
\begin{equation}\label{eq-2}
\xymatrix@C=40pt@R=30pt{
\widetilde X  \ar[rr]^\pi \ar[rd]_{\ \ \ \varphi}  &  & X \ar[ld]^{\phi \ \ \ }\\ 
&   Y. &  \\   
}
\end{equation}

Compared to the smooth case, the proof of the following claim requires non-trivial arguments arising from singularities, 
which are solved by properties of the almost nefness of $\mathcal{T}_{X}$. 

\begin{claim}\label{claim-almost-nef-fibre}
A very general fiber of $\phi \colon  X \to Y$ has almost nef tangent sheaf. 
\end{claim}

\begin{proof}[Proof of Claim \ref{claim-almost-nef-fibre}]
Let $F$ be a very general fiber of $\phi \colon  X \to Y$, and let   $\widetilde{F}$ be  the fiber of $\varphi \colon  \widetilde{X} \to Y$ at the same point as $F$. 
Note that $\widetilde{F}=\pi^{-1}(F)$. 
Since $F$ is a very general fiber, 
the restriction $\pi|_{\widetilde{F}} \colon \widetilde{F} \to F$ is 
a resolution of singularities of $F$. 
Hence, we can take  a big Zariski open subset $F_{0} \subseteq F$ such that 
$$\pi|_{\widetilde{F}_{0}}\colon \widetilde{F}_{0}\coloneqq \pi^{-1}(F_{0}) \to F_{0}$$ is isomorphic. 
Restricting the standard sequence of tangent bundles associated with $\varphi \colon  \widetilde{X} \to Y$ to $\widetilde{F}$, 
we obtain the left exact sequence 
$$
0\to \mathcal{T}_{\widetilde{F}} \to \mathcal{T}_{\widetilde{X}}|_{\widetilde{F}}\to (\varphi^{*} \mathcal{T}_{Y})|_{\tilde{F}}\cong\mathcal{O}_{\widetilde{F}}^{\oplus k}\to 0, 
$$
where \(k=\dim F\). 
By the push-forward along $\pi|_{\widetilde{F}}$, we obtain the exact sequence of sheaves  
$$
0\to (\pi|_{\widetilde{F}})_*\mathcal{T}_{\widetilde{F}} \to (\pi|_{\widetilde{F}})_*( \mathcal{T}_{\widetilde{X}}|_{\widetilde{F}} )\to \mathcal{O}_{F}^{\oplus k}. 
$$
Note that the sheaf morphism $(\pi|_{\widetilde{F}})_*( \mathcal{T}_{\widetilde{X}}|_{\widetilde{F}} ) \to \mathcal{O}_{F}^{\oplus k}$ is generically surjective
since $\pi|_{\widetilde{F}}\colon  \widetilde{F} \to F$ is isomorphic over $F_{0}$. 
Since $F$ is a very general fiber, the restriction $\mathcal{T}_X|_F$ of the tangent sheaf $\mathcal{T}_{X}$
is reflexive. 
Furthermore, we have 
$$(\pi|_{\widetilde{F}})_*( \mathcal{T}_{\widetilde{X}}|_{\widetilde{F}} )=\mathcal{T}_{X}|_{F}
\text{ and } (\pi|_{\widetilde{F}})_*\mathcal{T}_{\widetilde{F}}=\mathcal{T}_F \text{ on } F_{0}. $$  
Thus, by reflexivity, we obtain 
$$
((\pi|_{\widetilde{F}})_*( \mathcal{T}_{\widetilde{X}}|_{\widetilde{F}} ))^{\vee\vee}=\mathcal{T}_X|_F \text{ and } ((\pi|_{\widetilde{F}})_*\mathcal{T}_{\widetilde{F}})^{\vee\vee}=\mathcal{T}_F \text{ on } X. 
$$
Then, by taking the double dual of the above exact sequence, we obtain the left exact sequence
\[
0\to \mathcal{T}_F\to \mathcal{T}_X|_F\to \mathcal{O}_{F}^{\oplus k}.
\]
The morphism $\mathcal{T}_X|_F\to \mathcal{O}_{F}^{\oplus k} \to \mathcal{O}_{F}$ induced by the projection 
$\mathcal{O}_{F}^{\oplus k} \to \mathcal{O}_{F}$ is generically surjective.
Since $F$ is a very general fiber, the restriction  \(\mathcal{T}_X|_F\) is almost nef. 
Hence,  the sheaf morphism $\mathcal{T}_X|_F\to \mathcal{O}_{F}^{\oplus k}$ is actually surjective by Lemma \ref{lem-elementary-lem-almostnef} (6).  
On the big Zariski open set \(F_0\), there exists the surjective morphism
$$(\bigwedge^{k+1}\mathcal{T}_X|_F)|_{F_0}\to \mathcal{T}_{F_0},$$ 
which determines the generically surjective morphism
$\bigwedge^{[k+1]}\mathcal{T}_X|_F\to \mathcal{T}_F$
on  $F$. 
Hence, the tangent sheaf  $\mathcal{T}_F$ is almost nef by Lemma \ref{lem-elementary-lem-almostnef} (4). 
\end{proof}

We are left to show Theorem \ref{thm-almost-nef} (1) and (3).
Both properties (1) and (3) directly follow from \eqref{eq-2} regardless of the almost nefness of $\mathcal{T}_{X}$. 
Therefore,   we summarize the proofs in an independent form at the end of this section (see Lemma \ref{lem-flat-reduced}). 
\end{proof}

\begin{lem}
\label{lem-flat-reduced}
Let $\phi \colon  X \to Y$ be a fibration from a projective klt variety $X$ onto a smooth projective variety $Y$.
Let $\pi \colon  \widetilde{X} \to X$ be a log resolution of $X$. 
Assume there exists a smooth fibration $\varphi \colon  \widetilde{X} \to Y$ such that $\varphi=\phi \circ \pi$ $($cf.\,\eqref{eq-2}$)$. 
Then, we have$:$
\begin{enumerate}
\item[$(1)$]  $ \phi \colon X \to Y$ is a flat fibration. 
\item[$(2)$] Any fiber of  $\phi \colon X \to Y$ is reduced and irreducible. 
\item[$(3)$] Every irreducible component of the singular locus of $X$ dominates over $Y$. 
\item[$(4)$]  Any fiber of \(\phi \colon X\to Y\) has only klt singularities.
\end{enumerate}

\end{lem}

\begin{proof}

(1). By the miracle flatness, it is sufficient to check that $ \phi$ is an equidimensional fibration 
by noting \(X\) is Cohen-Macaulay and \(Y\) is smooth. 
Let $a$ (resp.\,$y$) be an arbitrary point (resp.\,a general point) in $Y$. 
Let $X_{a}$ (resp.\,$\widetilde{X}_{a}$) denote the fiber of $ \phi \colon X \to Y$ at $a \in Y$ 
(resp.\,the fiber of $\varphi \colon  \widetilde{X} \to Y$ at $a \in Y$). 
By Chevalley's theorem, we obtain 
\[\dim (\widetilde{X}_y)=\dim(X_y)\le \dim (\widetilde{X}_a)=\dim (\widetilde{X}_y).\]
Here, the equality on the right-hand side follows from the equidimensionality of $\varphi$. 
This indicates that $ \phi \colon X \to Y$ is a flat fibration.

\smallskip

(2). Since any fiber of $\varphi \colon  \widetilde{X} \to Y$ is irreducible, so is any fiber of $\phi \colon  X \to Y$. 
We now check that an arbitrary fiber \(X_a\) is reduced. 
For this purpose, we suppose that the fiber \(X_a\) is non-reduced for some point $a \in Y$. 
The conclusion is local in $Y$. 
By taking a complete intersection, we may assume $Y$ is a unit disk containing \(a\).
Let $m \geq 2$ be the multiplicity of the fiber $X_{a}$. 
Take a cover $\tau \colon  Y' \to Y$ ramified at $a' \in Y'$ such that $\tau(a')=a$ holds and the ramification index at \(a'\in Y'\) is $m$. 
Let $\overline{X}$ be the normalization of the fiber product $X \times_{Y}Y'$. 
Then, by Lemma \ref{lem-ramification-on-curve} below, 
the fiber $\overline{X}_{a'}$ of $\overline{X} \to Y'$ at $a'$ 
is a reducible variety with $m$ distinct components. 
Meanwhile, the fiber product $\widetilde{X}' \coloneqq  \widetilde{X} \times_{Y} Y'$ admits the morphism 
\(\widetilde{X}' \to \overline{X}\) by the universality of normalization. 
Then, we have the following commutative diagram$:$
\[
\xymatrix{
\widetilde{X}' \ar[r]\ar[d]&\overline{X}\ar[r]\ar[d]& Y'\ar[d]^{\tau} \\
\widetilde{X}\ar[r]^{\pi}&X\ar[r]^{\phi}&Y. 
}
\]
The fiber \(\widetilde{X}'_{a'}\) of $\widetilde{X}' \to Y'$ at $a' \in Y'$ is smooth and connected, 
but it is mapped onto $m$ distinct components of \(\overline{X}_{a'}\). 
This contradicts $m \geq 2$.

\smallskip
(3). We first show that  an arbitrary fiber \(X_a\)  of $\phi \colon X \to Y$ is not contained in the singular locus $X_\sing$ of \(X\). 
Suppose \(X_a \subseteq X_{\sing}\) holds for some $a \in Y$.
Let $Z$ be an irreducible component of $X_\sing$ containing \(X_a\) and 
let \(E\subseteq\widetilde{X}\) be a prime divisor dominating \(Z\). 
By Chevalley's theorem, any fiber of \(\pi|_E\) has dimension \(\ge 1\) 
by noting that \( \dim (X) -1=\dim(E)> \dim(Z)\) holds. 
Hence, by $\widetilde{X}_a = \pi^{-1}(X_{a})$, we obtain  
\[
\dim (X_a)=\dim (\widetilde{X}_a)\ge \dim (\pi|_E)^{-1}(X_a)\ge \dim (X_a)+1. 
\]
This is a contradiction. 

It remains to show that every irreducible component of the  \(\pi\)-exceptional locus dominates over \(Y\).
Suppose that there exists a \(\pi\)-exceptional prime divisor \(E\) such that $\varphi(E)$ is properly contained in \(Y\). 
Then, the image \(\varphi(E) \subseteq Y\) is an irreducible subvariety of codimension one. 
Since \(\varphi \colon  \widetilde{X} \to Y\) is smooth, the inverse image \(\varphi^{-1}(\varphi(E))\) is reduced and irreducible. 
Therefore, we obtain \(E=\varphi^{-1}(\varphi(E))\) by \(E\subseteq\varphi^{-1}(\varphi(E))\). 
Then, for a point \(a\in \varphi(E)\), the fiber \(X_a\) is contained in the singular locus of \(X\), 
which contradicts the first part of the proof of (3).

\smallskip
(4). Let $F$ denote an arbitrary fiber of $\phi \colon X \to Y$. 
We first show that \(F\) is normal.
Since \(\phi\) is flat by (1) and \(X\) is Cohen-Macaulay,  the fiber \(F\) is Cohen-Macaulay. 
Hence, we are left to show that \(F\) is smooth in codimension one.
For this purpose, it is sufficient to show that \(\text{codim}_F(\pi(E_i\cap\widetilde{F}))\ge 2\) for any component \(E_i\). 
Here \(\widetilde{F}\) denotes the inverse image \(\pi^{-1}(F)\). 
Note that $\pi|_{\widetilde{F}}\colon \widetilde{F} \to F$ is a resolution of singularities of $F$. 
Suppose to the contrary that  \(E_1|_{\widetilde{X}_y}\) is not contracted to a subvariety of codimension \(\ge 2\) for some point \(y\in Y\). 
In the following, we assume \(E=E_1\) with \(\dim(E)>\dim(\pi(E))\), and let \(E_y=E\cap \widetilde{X}_y\).

By assumption, we have \(\dim(E_y)=\dim(\pi(E_y))\) for some point \(y\).
Then, for an ample divisor \(H\) on \(X\), the pullback  \(\pi^*H|_{E_y}\) is a nef and big divisor.
However, for a general fiber \(E_s\), since \(\dim(E_s)>\dim(\pi(E_s))\), the pullback \(\pi^*H|_{E_s}\) is not big.
Then, the following inequality gives rise to a contradiction.
\[
0=(\pi^*H|_{E_s})^{\dim (E_s)}=(\pi^*H)^{\dim (E_s)}\cdot E_s=(\pi^*H)^{\dim (E_y)}\cdot E_y=(\pi^*H|_{E_y})^{\dim (E_y)}>0.
\]
Hence,  we see that \(\text{codim}_F(\pi(E_i\cap\widetilde{F}))\ge 2\) for any component \(E_i\), 
and thus the fiber \(F\) is smooth in codimension one. In particular, the fiber \(F\) is normal.

Since each fiber is normal and Cohen-Macaulay,  we have  \(K_{X/Y}|_F=K_F\) for any fiber \(F\) by \cite[Proposition (9)]{Kle80}. 
When we write 
$$K_{\widetilde{X}/Y}=\pi^*K_{X/Y}+\sum a_iE_i, 
$$ 
we have \(a_i>-1\) by the klt condition.
Restricting this equality to \(\widetilde{F}\), we have
\(K_{\widetilde{F}}=(\pi|_{\widetilde{F}})^*K_F+\sum a_iE_i|_{\widetilde{F}}\).
This indicates that \(F\) is klt.
\end{proof}

\begin{lem}
\label{lem-ramification-on-curve}
Let \(\phi \colon X\to B\) be a flat fibration onto  a unit disk \(B \subseteq \mathbb{C}\) 
such that the central fiber \(X_{b}\) is an irreducible and non-reduced variety with multiplicity \(m\).
Let \(\tau \colon  B'\to B\) be a finite cover ramified at a point \(b'\in B'\) 
such that $\tau(b')=b$ and the ramification index at \(b'\in B'\) is $m$. 
Let \(X'\coloneqq X\times_BB'\) be the base change and \(\overline{X'}\) the normalization of \(X'\).
Then, the fiber \(\overline{X'}_{b'}\) at $b' \in B'$ is  reducible with \(m\) distinct components.
\end{lem}
\begin{proof}
Let \(n=\dim(X)\ge 2\).
For any smooth point \(x\in X_b\) of \(X\), we can locally embed \(X\) into \(\mathbb{C}^{n+1}\), in other words, around a small open neighborhood \(U_x\subseteq \mathbb{C}^{n+1}\) of \(x\), such that \(X\cap U_x\) is given by the zero locus of a polynomial in \(\mathbb{C}^{n+1}\). 
Fix a coordinate \(t\) on \(B\) and \((t,x_1,\cdots,x_n)\) on \(\mathbb{C}^{n+1}\). 
Denote by \(\tau_X \colon  X'\to X\) the induced ramified base change.
Then, locally around the small open subset \(\tau_X^{-1}(U_x)\subseteq\mathbb{C}^{n+1}\), the equation of \(X'\) is defined as
\[
s^m=f^m(x_1,\cdots,x_n),
\]
where \(s\) is the local parameter of \(B'\).
Here \(X'\cap \tau_X^{-1}(U_x)\) is given by the zero locus of the above equation \(s^m=f^m(x_1,\cdots,x_n)\) in \(\mathbb{C}^{n+1}\).
Then,  the variety \(X'\) is non-normal around \(X'_{b'}\) since locally there are \(m\) components (\(s-\zeta^if=0\) with \(\zeta^m=1\)) passing through the same fiber, and hence \(X'\) is not smooth in codimension one. 
After taking the normalization, we see that \(\overline{X'}_{b'}\) contains \(m\) distinct components.
\end{proof}

\section{Applications of the structure theorems}\label{Sec-6}

In this section, by applying Theorems \ref{thm-positively-curved} and \ref{thm-almost-nef}, 
we prove Theorem \ref{thm-MRC-Albanese-main}, Corollaries  \ref{cor-almost-nef}, \ref{cor-maximally-etale-main}.

\subsection{Proof of Theorem \ref{thm-MRC-Albanese-main}}\label{subsection-mrc-albanese}
This subsection is devoted to the proof of Theorem \ref{thm-MRC-Albanese-main} and Corollary \ref{cor-maximally-etale-main}. 
After reviewing the Albanese maps of projective klt varieties, we prove Lemma \ref{lem-agumented-irregularity-almostnef}. 
Let $X$ be a projective klt variety. 
The {\it irregularity}  $q(X)$ of $X$ is defined by $q(X) \coloneqq  \dim H^1(X, \mathcal{O}_X)$ and 
the {\it augmented irregularity} $\widehat{q}(X)$ is defined by 
$$
\widehat{q}(X) \coloneqq  \sup \{ q(\widehat{X})  \,|\, \text{$ \widehat{X} \to X$ is a finite quasi-\'etale cover} \}.
$$ 
For a resolution of singularities $\pi \colon  \widetilde{X} \to X$ of $X$, 
we consider the Albanese map $\widetilde{\alpha } \colon \widetilde{X} \to \Alb(\widetilde{X})$. 
Then,  by \cite[Corollary 1.6]{HM07} (cf.\,\cite[Lemma 8.1]{Kaw85}), 
we obtain the morphism $\alpha \colon  X \to \Alb(\widetilde{X}) $ 
with the following commutative diagram: 
 \[
 \xymatrix{
 \widetilde{X}\ar[r]^{\pi}\ar[dr]_{\widetilde{\alpha}}&X\ar[d]^{\alpha}\\
 &\Alb(\widetilde{X}).
 }
 \]
The Albanese map of $X$ is defined by $\alpha \colon  X \to \Alb(X)\coloneqq \Alb(\widetilde{X})$, 
which does not depend on the choice of resolutions of singularities (see \cite[Proposition 9.12]{Uen75}). 
Note that $\dim \Alb (X) = q(\widetilde{X}) = q(X)$ since $X$ has rational singularities.

\begin{lem}
\label{lem-agumented-irregularity-almostnef}
Let $X$ be a projective klt variety with pseudo-effective or almost nef tangent sheaf.
Then, the Albanese map $\alpha \colon  X \to \Alb(X)$ is surjective. 
In particular, the inequality $\widehat{q} (X) \le \dim X$ holds. 
Furthermore, there exists a maximally quasi-\'etale cover $\widehat{X} \to X$ with $q(\widehat{X})=\widehat{q}(X) $. 
 \end{lem}
\begin{proof} 
By Lemma \ref{lem-nonvanish}, any non-zero section of $H^0(X, \Omega_{X}^{[1]})$ is nonvanishing on $X_{\reg}$.
Thus, the same argument as in \cite[Proposition 3.9]{DPS94} shows that 
$\alpha|_{X_{\reg}} \colon  X_{\reg} \to \Alb(X)$ is smooth. 
Hence, the Albanese map $\alpha$ is surjective.

For the last assertion, we take a quasi-\'etale cover $X' \to X$ with $\widehat{q}(X) = q(X')$ by the first assertion, 
and consider a maximally quasi-\'etale cover $\widehat{X} \to X'$. 
Then, we obtain $$\widehat{q}(X) \geq q(\widehat{X}) \geq q(X') = \widehat{q}(X),$$ finishing the proof. 
\end{proof}

\begin{thm}
\label{thm-MRC-Albanese-almostnef}
Let $X$ be a projective klt variety with positively curved  or almost nef tangent sheaf.
Then, there exist a maximally quasi-\'etale cover $\widehat{X} \to X$ 
and a fibration $\varphi \colon  \widehat{X}  \to A$ onto  an abelian variety $A$ 
satisfying the following properties$:$
\begin{enumerate}
\item[$(1)$] $\varphi \colon  \widehat{X}  \to A$ is an everywhere-defined MRC fibration of $X$.
\item[$(2)$] $\varphi \colon  \widehat{X}  \to A$ is  the Albanese map with $\dim A = \widehat{q}(X) = q(\widehat{X})$.
\item[$(3)$] Consider the Fujita decomposition of $\mathcal{T}_{\widehat{X}}$ $($see\,Theorem \ref{thm-fujita-decomposition-klt}$)$$:$ 
\begin{equation*}
0 \longrightarrow  \mathcal{F} \longrightarrow \mathcal{T}_{\widehat{X}} \longrightarrow \mathcal{Q} \longrightarrow 0. 
\end{equation*}
Then, the reflexive hull $\mathcal{Q}^{\vee\vee}$, which is the flat part of \(\mathcal{T}_{\widehat{X}}\), coincides with $\varphi^{*}\mathcal{T}_{A}$.  
In particular, the augmented irregularity $\widehat{q}(X)$ is equal to the rank of  the flat part of $\mathcal{T}_{\widehat{X}}$.
\item[$(4)$] The augmented irregularity $\widehat{q}(X) $ is equal to the rank of the maximal destabilizing subsheaf of $\Omega_{X}^{[1]}$ with respect to $H_{1} \cdots H_{n-1}$ for any ample Cartier divisors $H_{1}, \ldots, H_{n-1}$. 
\end{enumerate}
\end{thm}

\begin{proof}

\smallskip 
(1).  By Lemma \ref{lem-agumented-irregularity-almostnef}, 
we can take a maximally quasi-\'etale cover $\widehat{X} \to X$ with $\widehat{q}(X) = q(\widehat{X})$. 
By Lemma \ref{lem-elementary-lem-almostnef} (3), 
the tangent sheaf $\mathcal{T}_{\widetilde X}$, which is the reflexive pullback of $\mathcal{T}_{X}$, is almost nef. 
By Theorems \ref{thm-positively-curved} and \ref{thm-almost-nef}, after taking a further quasi-\'etale cover of $\widehat{X}$, 
we obtain an everywhere-defined MRC fibration $\widehat{X}  \to A$ onto an abelian variety.

\smallskip 
(2). Take the Albanese map $\alpha \colon  \widehat{X} \to \Alb(\widehat{X})$.
By the universality of the Albanese map, there is a Lie group morphism $\beta \colon  \Alb(\widehat{X}) \to A$ 
with the following diagram (see \cite[Definition 9.6]{Uen75}).
$$
 \xymatrix{
 \widehat{X} \ar[r]^{\varphi} \ar[dr]_{\alpha}& A \\
    & \Alb(\widehat{X}) \ar[u]_{\beta}. \\
 }
$$
Since $\varphi$ is surjective, so is $\beta$. 
Then we obtain  $\dim A \le \Alb(\widehat{X})$.
On the other hand, since any fiber $F$ of $\varphi \colon \widehat{X} \to A$ is rationally connected, the image $\alpha(F)$ is a point. 
This indicates that $\dim A \ge \Alb(\widehat{X})$, 
and hence $\dim A = \dim \Alb(\widehat{X})$. 
In particular, we see that $\beta \colon  \Alb(\widehat{X}) \to A$ is a finite \'etale cover.

Since  $\varphi^{-1}(a)$ is irreducible for any $a \in A$ and \(\alpha \colon \widehat{X} \to A\) has connected fibers, 
the fiber $\beta^{-1}(a)$ is a single point. 
This indicates that $\beta \colon \Alb(\widehat{X}) \to A$ is isomorphic and $\varphi \colon  \widehat{X}  \to \Alb(\widehat{X}) $ is the Albanese map.
Moreover, we obtain $\dim A = q(\widehat{X}) = \widehat{q}(X)$.

\smallskip 
(3). We prove that $\mathcal{Q}^{\vee} = \varphi^{*}\Omega_{A}$. 
For this purpose, we first check that  $\dim A  = \rk \mathcal{Q}$ holds.
The pullback $\varphi^{*}\Omega_{A} \subseteq \Omega_{\widehat{X} }^{[1]}$ 
is a flat vector bundle, and hence we obtain $\dim A \le \rk \mathcal{Q}$
since $\mathcal{Q}^{\vee}$ is the maximal flat locally free sheaf contained in $\Omega_{\widehat{X} }^{[1]}$ 
by Remark \ref{rem-maximal-flat-locally-free}. 
To check the converse inequality, for a general fiber $F$ of $\varphi$, we consider the following diagram:
\begin{equation*}
\label{eq-fujita-exact-sq}
\xymatrix@C=25pt@R=20pt{
0 \ar@{->}[r]&  \mathcal{F}|_{F} \ar@{->}[r] &\mathcal{T}_{\widehat{X}}|_{F}  \ar@{->}[r]\ar@{=}[d]&\mathcal{Q}|_{F}\ar@{->}[r] &0\\
0\ar@{->}[r] &\mathcal{T}_{F}  \ar@{->}[r] \ar@{-->}[rru]_{\,\,\, \gamma}&\mathcal{T}_{\widehat{X}}|_{F} \ar@{->}[r]& \mathcal{O}_{F}^{\oplus \dim A}.&
}
\end{equation*}
We claim that the induced morphism $\gamma \colon  \mathcal{T}_{F} \to \mathcal{Q}|_{F}$ 
is actually the zero map.
Indeed, we see that the general fiber is simply connected by \cite[Theorem 1.1]{Tak03}, by noting that
 $F$ is a rationally connected projective klt variety. 
This implies that the flat locally free sheaf $\mathcal{Q}|_{F}$ is a trivial vector bundle. 
On the other hand, we have $H^0(F, \Omega_{F}^{[1]}) =0$ due to rational connectedness of \(F\). 
Hence $\gamma $ is the zero map, and thus, we can obtain an injective morphism $\mathcal{T}_{F} \to \mathcal{F}|_{F}$. 
In particular, we obtain
$$
\dim \widehat{X} - \dim A = \rk \mathcal{T}_{F} \le \rk \mathcal{F}|_{F} = \dim \widehat{X} - \rk \mathcal{Q}.
$$
Thus we have $\dim A = \rk \mathcal{Q}$.

Since both $\varphi^{*}\Omega_{A}$ and $\mathcal{Q}^{\vee\vee}$ are almost nef locally free sheaves, by applying Lemma \ref{lem-nonvanish} (2) to the morphism $0 \to \varphi^{*}\Omega_{A} \to \mathcal{Q}^{\vee}$, the subsheaf $\varphi^{*}\Omega_{A}|_{X_{\reg}} $ is a subbundle of $\mathcal{Q}^{\vee}|_{X_{\reg}}$.
This implies $\mathcal{Q}^{\vee}$ coincides with $\varphi^{*}\Omega_{A}$ on the big open set $X_{\reg}$. Since both are reflexive, these two sheaves coincide.

\smallskip 
(4). This follows from $\dim A = \rk \mathcal{Q}$ and  Proposition \ref{prop-genericallyample-etale}.
\end{proof}

The following corollary is an immediate consequence of Theorem \ref{thm-MRC-Albanese-almostnef}, 
which provides a relation between the positivity of tangent sheaf and the augmented irregularity. 
We refer to Examples \ref{exa-Qabelian} and \ref{exa-Matsuzawa-Yoshikawa} for some practical calculation.

\begin{cor}\label{cor-maximally-etale-rc}
Let \(X\) be a projective klt variety with positively curved or almost nef tangent sheaf.
Then, the following conditions are equivalent.
\begin{enumerate}[label=$(\alph*)$]
\item The tangent sheaf $\mathcal{T}_X$ is generically ample.
\item The augmented irregularity $\widehat{q}(X)$ is zero.
\item Any finite quasi-\'etale cover of $X$ is rationally connected.

\end{enumerate}
Furthermore, if $X$ is maximally quasi-\'etale, then the above equivalent conditions are also equivalent to the following conditions$:$
 \begin{enumerate}[label=$(\alph*)$]\setcounter{enumi}{3}
 \item $X$ is rationally connected.
 \item $X$ is simply connected.
 \end{enumerate}
\end{cor}

\begin{proof}
By Theorem \ref{thm-MRC-Albanese-almostnef} (4), we have already known $(a) \Leftrightarrow (b)$. 
Since the irregularity of rationally connected varieties is zero, we have $(c) \Rightarrow (b)$.
 We now show that if $(c)$ does not hold, then $(b)$ does not.
 If some quasi-\'etale cover $X'$ of $X$ is not rationally connected, by Theorem \ref{thm-MRC-Albanese-almostnef} (1), (2), we can take a quasi-\'etale cover \(\widehat{X'}\to X'\) and the corresponding MRC  fibration \(\widehat{X'}\to A\) with \(\dim(A)=\widehat{q}(X')\). 
 Since $\widehat{X'}$ is not rationally connected,  we obtain $\widehat{q}(X) = \widehat{q}(X') = \dim A \neq 0$, contradicting $(b)$.

For the last conclusion for $(d)$ and $(e)$, we only prove that $(e) \Rightarrow (a)$ under the assumption that $X$ is maximally quasi-\'etale (cf.~\cite[Section 2]{Wu21}). 
Suppose to the contrary that there exist ample divisors \(H_1,\ldots, H_{n-1}\) with \(\mu^{\min}_\alpha(\mathcal{T}_X)=0\), 
where \(\alpha=H_1\cdots H_{n-1}\).
Consider the Harder-Narasimhan filtration
\[
0=\mathcal{E}_0\subseteq \mathcal E_1\subseteq\cdots\subseteq \mathcal E_r=\mathcal T_X. 
\]
Then, we obtain \(\mu^{\min}_\alpha (\mathcal T_X) =\mu_{\alpha}(\mathcal T_X/ \mathcal E_{r-1})=0\). 
Since \(\mathcal T_X\) is positively curved or almost nef, so is \(\mathcal{Q}\coloneqq \mathcal T_X/ \mathcal E_{r-1}\). Hence \(\mathcal{Q}^{\vee\vee}\) is a flat locally free sheaf by Proposition \ref{prop-almostnef-c2-locallyfree}.
Thus, we obtain an injective sheaf morphism \(\mathcal{Q}^{\vee}\to \Omega_X^{[1]}\).
Since any flat locally free sheaf is constructed by a linear representation of the fundamental group, we have \(\mathcal{Q}^{\vee}\cong \mathcal{O}_{X}^{\oplus \rk \mathcal{Q}}\) by the triviality of \(\pi_1(X)\). In particular, we obtain $H^0(X, \Omega_{X}^{[1]}) \neq 0$.
On the other hand, for any resolution \(\pi \colon \widetilde{X}\to X\), \cite[Theorem 1.4]{GKKP11} gives rise to the equality
$$
\dim_{\C} H^0(X, \Omega_{X}^{[1]}) 
= \dim_{\C} H^0(\widetilde{X}, \Omega_{\widetilde{X}}).
$$
Since $\widetilde{X}$ is also simply connected by \cite[Theorem 1.1]{Tak03}, the right-hand side is zero, which is absurd.
\end{proof}

\begin{ex}
\label{exa-Qabelian}
By \cite{Uen75} and \cite{Cam11}, 
there exists a klt rationally connected $Q$-abelian threefold $X$.
Since $X$ admits a quasi-\'etale cover $A \to X$ from an abelian variety $A$, 
the tangent sheaf $\mathcal{T}_{X}$ is positively curved by Lemma \ref{lem-p-curved}.
Furthermore, we have $\widehat{q}(X)= \rk \Omega_{X}^{[1]}=3.$
In this example, the variety $X$ is rationally connected, but some quasi-\'etale cover is not rationally connected.
By Corollary \ref{cor-maximally-etale-rc}, the variety $X$ is not maximally quasi-\'etale.
\end{ex}

The following example constructed in \cite[Section 7]{MY21} is a typical example in the study of non-isomorphic surjective endomorphisms.

\begin{ex}[{cf.\,\cite[Section 7]{MY21}}]
\label{exa-Matsuzawa-Yoshikawa}
Let \(E\) be an elliptic curve and \(Y\coloneqq \mathbb{P}(\mathcal{O}_E\oplus\mathcal{L})\) be a projective bundle where \(\mathcal{L}\) is a invertible sheaf of degree zero.
Note that the case $\mathcal{L}=\mathcal{O}_{E}$ has been studied in \cite[Subsection 3.B]{GKP14} (cf.\,\cite[Example 1.1]{Ou14}).
Let $[-1] \colon E \to E$ be the multiplication by $(-1)$ map.
Then, we obtain the following isomorphisms:
\begin{itemize}
\item $\gamma \colon  \mathbb{P}(\mathcal{O}_E\oplus [-1]^{*}\mathcal{L}) \to Y$ from pullback of $\mathcal{O}_E\oplus\mathcal{L}$ by $[-1]$.
\item $\Phi \colon    \mathbb{P}(\mathcal{O}_E\oplus \mathcal{L}^{-1}) \to \mathbb{P}(\mathcal{O}_E\oplus [-1]^{*}\mathcal{L})$ from $\mathcal{L}^{-1} \cong [-1]^{*}\mathcal{L}$.
\item $\iota \colon    Y \to \mathbb{P}(\mathcal{O}_E\oplus \mathcal{L}^{-1}) $ from 
$\mathcal{O}_E\oplus \mathcal{L}^{-1} \cong \mathcal{L}^{-1} \oplus \mathcal{O}_E \cong (\mathcal{O}_E \oplus\mathcal{L}) \otimes \mathcal{L}^{-1}$.
\end{itemize}
Then, the involution \(\sigma \coloneqq  \gamma \circ \Phi \circ \iota \)  is isomorphism of $Y$ over $E$.
Thus, we can take the quotient \(X\coloneqq Y/\langle  \sigma \rangle \)  satisfying the following commutative diagram$:$
 \begin{equation}
 \label{eq-Matsuzawa-Yoshikawa}
\xymatrix@C=40pt@R=30pt{
Y\ar[r]^(0.32){\nu}\ar[d]_{\varphi}&X=Y/\langle  \sigma \rangle \ar[d]^{f}\\
E \ar[r]_(0.32){\rho} &E/\langle  -1\rangle  \cong\mathbb{P}^1, 
 }
 \end{equation}
where the horizontal arrows are quotient morphisms and \(f\) is induced from \(\varphi\).
Furthermore, the morphism \(\nu \colon Y \to X\) is quasi-\'etale and \(-K_X\) is nef.
Moreover, if \(\mathcal{L}\) is non-torsion (resp.\,torsion), then \(\kappa(-K_X)=0\) (resp.\,\(\kappa(-K_X)=1\)) holds.
Now, we claim the following$:$
\begin{claim}
\label{claim-MY-example}
\begin{enumerate}
\item $X$ is a rational surface with canonical singularities. More precisely, the variety $X$ has 8 singular points with $A_1$ singularities. 
\item $\mathcal{T}_X$ is positively curved.
\item There exists a $\Q$-Cartier divisor $D$ with $2 D \sim 0$ and $\mathcal{O}_{X} (D) \subseteq \Omega_{X}^{[1]}$. 
In particular, the tangent sheaf $\mathcal{T}_{X}$ is not generically ample.
\item \(\mathcal{T}_X\) is not almost nef.
\end{enumerate}
\end{claim}
Since $\varphi$ is an MRC  fibration (and also the Albanese map), we obtain 
$
\widehat{q}(X) = \widehat{q}(Y) =1.
$
Hence, by Theorem \ref{thm-MRC-Albanese-almostnef}, 
there exists a maximal destabilizing sheaf in $ \Omega_{X}^{[1]}$ of rank one, which is just $\mathcal{O}_{X}(D)$.
By Corollary \ref{cor-maximally-etale-rc}, the variety $X$ is not maximally quasi-\'etale, 
since $X$ is rationally connected while $Y$ is not rationally connected.

\begin{proof}[Proof of Claim \ref{claim-MY-example}]
(1). We show that $X$ has canonical singularities.  
Take ramification points $p_1, \ldots, p_4 \in E$ of $\rho \colon  E \to \mathbb{P}^1$.
By the construction of the involution $\sigma \colon  Y \to Y$,
the restriction $\sigma|_{\varphi^{-1}(p_i)}$ is expressed as follows:
\begin{equation*}
\label{eq-MY-intamplified}
\begin{array}{ccccc}
\sigma|_{\varphi^{-1}(p_i)} \colon  &Y|_{\varphi^{-1}(p_i)} = \mathbb{P}^1& \longrightarrow & Y|_{\varphi^{-1}(p_i)} = \mathbb{P}^1 \\
&(z: w) & \longmapsto & (w : z).
\end{array}
\end{equation*}
Thus, there are 8 fixed points of $\sigma$ that are singular points of $X$ with $A_1$ singularity. 
In particular, the variety $X$ has canonical singularities.

(2). It is sufficient to show that $\mathcal{T}_Y$ is positively curved (cf.~Lemma \ref{lem-p-curved}). 
Notice that the universal cover of $Y$ is isomorphic to $\C \times \mathbb{P}^1$.
Since the Euclidean space $\C$ has the standard Euclidean metric $g_{\C}$ and the projective space $\mathbb{P}^1$ has the Fubini-Study metric $g_{FS}$, 
the tangent bundle $\mathcal{T}_Y$ admits a (smooth) positively curved metric $g_{Y}$, which is a quotient metric of $g_{\C} \times g_{FS}$ by the action of $\pi_{1}(E)$. 
Indeed, the metric $g_{Y}$ is a K\"ahler metric with semi-positive biholomorphic sectional curvature.

(3). 
We follow the argument in \cite[Example 3.7]{GKP14}.
Take the branch points $q_1, \ldots, q_4 \in \mathbb{P}^1$ of $\rho \colon  E \to \mathbb{P}^1$.
Let $F_{i} \coloneqq f^{-1}(q_i)$ be the set-theoretic fiber of $q_i$ and $F$ a general fiber of $f$.
Consider the  differential map 
$$
d f|_{X_{\reg}} \colon  \pi^{*} \Omega^{1}_{\mathbb{P}^1}|_{X_{\reg}} \to \Omega^{1}_{X_{\reg}}. 
$$
Since $f^{*}[q_i] = 2 F_{i}$ by the commutative diagram, we obtain the following factorization$:$
 \begin{equation*}
\xymatrix@C=30pt@R=30pt{
\pi^{*} \Omega^{1}_{\mathbb{P}^1}|_{X_{\reg}}\ar@{->}@/^18pt/[rr]^{d f|_{X_{\reg}}} \ar[r]&
\mathcal{O}_{X_{\reg}} (-2 F + \sum_{i=1}^{4} F_{i} \cap X_{\reg}) \ar[r]
 & \Omega^{1}_{X_{\reg}}.\\
 }
 \end{equation*}
Thus, by taking a reflexive hull, the divisor $D \coloneqq  -2 F + \sum_{i=1}^{4} F_{i}$ satisfies $\mathcal{O}_{X} (D) \subseteq \Omega_{X}^{[1]}$.
Hence, we conclude our proof by the following linear equivalence.
$$
2 D =-4 F + \sum_{i=1}^{4} 2F_{i} \sim -4 F + \sum_{i=1}^{4}f^{*}[q_i] \sim 0.
$$

(4). Let \(\pi \colon \widetilde{X}\to X\) be the minimal resolution with the induced morphism \(g \colon \widetilde{X}\to C\coloneqq \mathbb{P}^1\) such that \(g=f\circ\pi\). 
Since \(\widetilde{X}\) is the minimal resolution of the 8 \(A_1\)-singularities, there are 4 singular fibers of \(g\) such that 
$$g^*p_i=2\ell_i+M_i+N_i,$$ 
where \(M_i\) and \(N_i\) are two disjoint \((-2)\)-curves and \(\ell_i\) is a \((-1)\)-curve such that \(\ell_i\cdot M_i=\ell_i\cdot N_i=1\) (cf.\,\cite[Figure 3.1]{GKP14}).
In the following, we follow the notation as in \cite[Notation 2.2]{JLZ23}.
Applying \cite[Section 3]{Ser96}, we have a natural surjection
\[
T_{\widetilde{X}}\to (-g^*K_C-E)\otimes\mathcal{I}_{\Gamma}\to 0, 
\]
where \(E=\sum_{c\in C}g^*c-(g^*c)_{\text{red}}=\sum_{i=1}^4\ell_i\) and \(\mathcal{I}_\Gamma\) is an ideal sheaf and the support of \(\Gamma\) consists of points \(b\in\widetilde{X}\) such that the reduced structure \(g^{-1}(g(b))_{\text{red}}\) is singular at \(b\).
Since the local equation of \(g^{*}(p_i)\) is of the form \(x^{\mu_i}y^{\nu_j}=0\),  
for each \(b\in \text{Supp}\,\Gamma\), the stalk \(\mathcal{I}_{\Gamma,b}=(x,y)\) is reduced by \cite[Proof of Proposition 3.1]{Ser96} (cf.~\cite[Table 1]{JLZ23}).
Then, we have
\[
-g^*K_C-E\sim_{\mathbb{Q}}\frac{1}{2}\sum_{i=1}^4(M_i+N_i).
\]
Let 
$$
S\coloneqq\mathbb{P}_{\widetilde{X}}((-g^*K_C-E)\otimes\mathcal{I}_{\Gamma})\subseteq\mathbb{P}(T_{\widetilde{X}})$$ 
be a prime divisor, obtained via blow-ups of \(\mathcal{I}_{\Gamma}\) from \(\widetilde{X}\).
Denote by \(\chi \colon \mathbb{P}(T_{\widetilde{X}})\to \widetilde{X}\) the natural projection.
Write 
$$
(\chi|_S)^*M_i=M_i'+u_i \text{ and  }(\chi|_S)^*N_i=N_i'+v_i, $$
where \(u_i\) and \(v_i\) are \(\chi|_Y\)-exceptional.
Since \(\Gamma\) is reduced, the morphism \(S\to \widetilde{X}\) is a composite of smooth blow-ups. 
Then, by an elementary calculation (cf.~\cite[Notation 2.2 (9)]{JLZ23}), we obtain that 
\[\mathcal{O}_S(1)\coloneqq \mathcal{O}_{\mathbb{P}(T_{\widetilde{X}})}(1)|_S\sim_{\mathbb{Q}}\frac{1}{2}\sum_{i=1}^4(M_i'+N_i'-u_i-v_i).
\]
Note that the right-hand side is not pseudo-effective (see \cite[Lemma 4.12]{JLZ23}).
Therefore, by  \cite{BDPP13}, there exists a covering family of \(S\) such that \(\mathcal{O}_S(1)\) has a negative intersection with this family.
In particular, we see that \(S\subseteq\mathbb{S}(\mathcal{O}_{\mathbb{P}(T_{\widetilde{X}})}(1))\) that dominates \(\widetilde{X}\).
This implies that \(\mathcal{T}_{\widetilde{X}}\), and hence \(\mathcal{T}_X\),  is not almost nef (cf.~Lemma \ref{lem-functorial-resolution-tangent} and \cite[Proposition 1.2]{Wah75}).
\end{proof}
\end{ex}

\subsection{Proof of Corollary \ref{cor-almost-nef}}
This section is devoted to the proof of Corollary \ref{cor-almost-nef}. 
This shows that the almost nefness of tangent sheaves imposes a rather restrictive condition on the singularities.

\begin{proof}[Proof of Corollary \ref{cor-almost-nef}]
We first reduce the proof to the case where \(X\) is a rationally connected projective klt variety with \(K_X\sim_{\mathbb{Q}} 0\) 
and $\mathcal{T}_{X}$ is almost nef.
By Theorem \ref{thm-almost-nef}, there exists a finite \'etale cover \(X'\to X\) such that the Albanese morphism \(X'\to A\) is an MRC fibration.
Since \(K_{X'}\equiv 0\) and \(X'\) also has only klt singularities, 
by the abundance and \cite[Theorem 4.8]{Amb03},  
by replacing \(X'\) with  a further quasi \'etale cover, we may assume that \(X'\cong A\times F\), 
where \(F\) is a rationally connected fiber and \(A\) is an abelian variety. 
Thus, it is sufficient to show that $F$ is one point. 

We will show that  \(X\) is an \'etale quotient of an abelian variety 
under the assumption  that \(X\) is a rationally connected projective klt variety with \(K_X\sim_{\mathbb{Q}} 0\) and $\mathcal{T}_{X}$ is almost nef. 
Let \(X'\to X\) be the global index one cover with respect to \(mK_X\sim 0\) such that \(K_{X'}\sim 0\).
We claim that \(X'\) is smooth and hence \(X\) has only cyclic quotient singularities.
Indeed, since \(X'\) has only Gorenstein rational singularities,  it has only canonical singularities (cf.~\cite[Theorem 5.22]{KM98}).
Note also that \(\mathcal{T}_{X'}\) as the reflexive pullback of \(\mathcal{T}_X\) is still almost nef (cf.~Lemma \ref{lem-elementary-lem-almostnef} (4)). 
Let \(\widetilde{X'}\to X'\) be the functorial resolution as in Lemma \ref{lem-functorial-resolution-tangent}.
Then \(\mathcal{T}_{\widetilde{X'}}\) is almost nef and \(K_{\widetilde{X'}}\) is  pseudo-effective, 
by noting that \(X'\) has only canonical singularities.
Therefore, applying Corollary \ref{cor-c1-almostnef-tangent}, the variety \(\widetilde{X'}\) is an \'etale quotient of an abelian variety by the purity of branch loci (cf.~\cite[Corollary 4.7]{Iwa22}).
By Lemma \ref{lem-bircontr-Q-abelian},  the variety \(X'\cong\widetilde{X'}\) is smooth.

By the slicing theorem, we may assume that there is a cyclic group $G\coloneqq \langle g \rangle$  of order \(m\) 
acting on $\mathbb{C}^n$ such that locally we have $X=\mathbb{C}^n/G$.
Note that the natural linear representation \(\rho \colon G\to \text{GL}(\mathbb{C}^n)\) splits into a direct sum of one-dimensional representations. 
Hence, we may assume that the action of \(g\) on \(\mathbb{C}^n\) is 
via 
$$
g(z)=g(z_1,\cdots,z_n)=(e_{1/p_1}z_1,\cdots, e_{1/p_n}z_n),$$  
where $e_{1/k}\coloneqq \exp(2\pi i / k)$. 
Then, the group $G$  acts diagonally on 
$$
\mathcal{T}_{\mathbb{C}^n}\cong (\partial/ \partial z_1) \oplus\cdots\oplus (\partial/ \partial z_n)
\text{ by } g\cdot (\partial/ \partial z_k)_{(z_1,\cdots,z_n)}=e_{1/p_k}(\partial/ \partial z_k)_{g(z_1,\cdots,z_n)},  
$$ 
where \(k=1,2,\cdots,n\). 
Under the identification 
$$
\mathcal{T}_{\mathbb{C}^n}\cong (\partial/ \partial z_1) \oplus \cdots\oplus(\partial/ \partial z_n) \cong \mathcal{O}^{\oplus n},$$ 
the equivariant divisorial sheaf $\mathcal{T}_X \cong \mathcal{T}_{\mathbb{C}^n}^{G}$ (cf.~\cite[Appendix B]{GKKP11}) can be written as
$\mathcal{T}_X \cong \mathcal{T}_{\mathbb{C}^n}^{G} \cong J_1 \oplus \cdots\oplus J_n$, 
where 
\begin{align*}
J_k&=
\{f(z)(\partial/\partial z_k) \,|\, f(z)(\partial/\partial z_k)=e_{1/p_k} f(e_{1/p_1}z_1,\cdots, e_{1/p_n}z_n)(\partial/\partial z_k)\}\\
&=\{f(z) \,|\, f(z)=e_{1/p_k} f(e_{1/p_1}z_1,\cdots, e_{1/p_n}z_n)\}. 
\end{align*}
Note that $J_k$ is generated by $z_1^{a_1}\cdots z_n^{a_n}$ (as a $\mathbb{C}$-algebra)
such that 
$$(1/p_k)(1 + a_k) + \sum_{j\neq k}(1/p_j)a_j  \equiv_\mathbb{Z} 0.$$ 
On the other hand, the coordinate ring of $X$ can be regarded as ${\mathbb{C}[z_1,\cdots,z_n]}^G$, 
which is generated by $z_1^{a_1'}\cdots z_n^{a_n'}$ such  that $\sum_{j=1}^n(1/p_j)a_j'\equiv_\mathbb{Z} 0$. 
Let us consider $J_k^{\otimes p_k}$ and the morphism 
$$J_k^{\otimes p_k} \to \mathcal{O}_X={\mathbb{C}[z_1,\cdots,z_n]}^G.$$ 
This morphism is well-defined and the image is a non-trivial ideal. 
In fact, for \(p_k\) many  $(a_1^{(i)},\cdots,a_n^{(i)})$ with 
$$(1/p_k)(1 + a_k^{(i)}) + \sum_{j\neq k}(1/p_j)a_j^{(i)}  \equiv_\mathbb{Z} 0,$$ where \(i=1,2,\cdots,p_k\), 
the product $$\prod_{i=1}^{p_k} z_1^{a_1^{(i)}} \cdots z_n^{a_n^{(i)}}=z_1^{\sum_i a_1^{(i)}} \cdots z_n^{\sum_i a_n^{(i)}}$$ 
satisfies that 
\begin{align*}
(1/p_1)\sum_{i=1}^{p_k} a_1^{(i)} +\cdots  +(1/p_n) \sum_{i=1}^{p_k} a_n^{(i)} &\equiv_\mathbb{Z} 
(1/p_k)(p_k + \sum_{i=1}^{p_k} a_k^{(i)}) + \sum_{i=1}^{p_k} \sum_{j\neq k} (1/p_j)a_j^{(i)}\\
&=\sum_{i=1}^{p_k} ((1/p_k)(1+a_k^{(i)})+\sum_{j\neq k}(1/p_j)a_j^{(i)})
\equiv_\mathbb{Z} 0.
\end{align*}
This implies that the image is a non-trivial  element of $\mathcal{O}_X={\mathbb{C}[z_1,\cdots,z_n]}^G$. 
From this, we obtain  the morphism $J_k^{\otimes p_k} \to \mathcal{O}_X={\mathbb{C}[z_1,\cdots,z_n]}^G$. 
In particular, the sheaf \(J_k^{\otimes p_k}\) is not almost nef.
By Lemma \ref{lem-elementary-lem-almostnef} (2) and (5), we see that \(J_k\) is not almost nef.
We note that \(X\) is smooth if and only if \(G\) has no fixed points on \(\mathbb{C}^n\) if and only if \(J_k^{\otimes p_k}\) is trivial for each \(k\) (recalling that here our argument is local).
Since we have the splitting \(\mathcal{T}_X=\oplus_k J_k\) (cf.~\cite[Appendix B]{GKKP11}), it follows that \(\mathcal{T}_X\) is not almost nef once \(G\) has a fixed point.
\end{proof}

As an immediate consequence of Corollary \ref{cor-almost-nef}, we deduce the following result.

\begin{cor}\label{ques-high-dim-not-almost-but=pseudo-effective}
Let \(X\) be a normal projective variety, which is a quasi-\'etale quotient of an abelian variety.
Then the tangent sheaf \(\mathcal{T}_X\) is pseudo-effective but not almost nef unless \(X\) is smooth. 
\end{cor}

\section{Examples of projective varieties with certain positive tangent sheaf}\label{Sec-7}

In this section, we provide examples of projective varieties with certain positive tangent sheaf and with vanishing augmented irregularity. 
Furthermore, we classify projective varieties with almost nef tangent sheaf  in dimensions two and three.

\subsection{Varieties with vanishing augmented irregularity}

The first example is a toric variety, which is a typical example whose tangent sheaf satisfies a certain positivity.

\begin{ex}\label{ex-almost-nef-toric}
Let \((X,D)\) be a projective toric pair.
Then, the tangent sheaf $\mathcal{T}_{X}$ is positively curved  and almost nef. 
Indeed, there exists a big open subset \(U\subseteq X\) 
such that  \(U\) is smooth, \(D|_U\) is normal crossing, \(\Omega_U(\text{log}\,D|_U)\) is free (see \cite[4.3, page 87]{Ful93}). 
Then, considering the dual of the injective sheaf morphism \(\Omega_U\to\Omega_U(\text{log}\,D)\), 
we obtain a generically surjective morphism \(\mathcal{O}_X^{\oplus\dim X}\to\mathcal{T}_X\).
Thus, the tangent sheaf \(\mathcal{T}_X\) is 
almost nef by Lemma \ref{lem-elementary-lem-almostnef} (4), and is positively curved 
by noting that the positively curved quotient metric over \(U\) can be extended to \(X\).

Note that if $X$ is \(\mathbb{Q}\)-Gorenstein, then $X$ has vanishing augmented irregularity, 
since it is of Fano type (see \cite[Theorem 1.13]{GKP16}). 
\end{ex}

\begin{ex}\label{ex-gras}
Let \(X\) be a Grassmannian. 
Then, the variety $X$ is Fano, and hence rationally connected. 
Furthermore, since \(X\) is homogeneous, the tangent bundle \(\mathcal{T}_X\) is globally generated. 
However, the variety $X$ is not necessarily  toric. 
This is a trivial example, but it tells us that the converse direction of Example \ref{ex-almost-nef-toric} is not true.
\end{ex}

Toric varieties are almost homogeneous. 
From this perspective, we  extend Example \ref{ex-almost-nef-toric} slightly to the following example: 

\begin{ex}[{cf.~\cite[Example 1]{Wu21}}]\label{ex-almost-homo}
Let \(G\) be a linear algebraic group acting on a normal projective variety \(X\). 
Assume that $X$ is {\textit{ \(G\)-almost homogeneous}} (i.e.,\,there exists an open dense \(G\)-orbit \(U\) in \(X\)). 
Then, by $U \subseteq  X_{\reg}$, 
a \(G\)-equivariant functorial resolution \(\pi \colon \widetilde{X}\to X\) is isomorphic over $U$. 
Since \(\widetilde{X}\) is also \(G\)-almost homogeneous, the tangent sheaf \(\mathcal{T}_{\widetilde{X}}\) is globally generated over \(U\). 
In particular, by  Lemma \ref{lem-p-curved}, the sheaf \(\mathcal{T}_{\widetilde{X}}\) is positively curved, and so is \(\mathcal{T}_X\). 
\end{ex}

Motivated by Corollary \ref{cor-almost-nef}, Claim \ref{claim-MY-example}, and Examples \ref{ex-almost-nef-toric}, \ref{ex-gras}, and \ref{ex-almost-homo}, 
we ask the following question (cf.~\cite[Problem 3.12]{HIM22}), which predicts the geometry of a general fiber of the structure theorems.

\begin{ques}[{cf.~Corollary \ref{cor-maximally-etale-main}}]\label{ques-anf-vanishing-aug}
Let \(X\) be a rationally connected projective klt variety with almost nef tangent sheaf.
Is the augmented irregularity of \(X\) zero? 
Is the anti-canonical divisor \(-K_X\) big?
\end{ques}

In what follows, we explore various forms of positivity for tangent sheaves and compare them to the case of smooth varieties. 
We first extend the main results of \cite{Sie17} and \cite{LiOY18} to the following proposition on the strictly nef tangent sheaf 
(see also Remark \ref{rmk-v-big-singular}).

 \begin{prop}[{Varieties with strictly nef tangent sheaves}]\label{prop-strictly-nef-tangent}
 Let \(X\) be a normal projective variety of dimension \(n\).
  If $\mathcal{T}_{X} $ is strictly nef, then $X \cong \mathbb{P}^n$.  
 \end{prop}
 \begin{proof}
 Let \(\pi \colon \widetilde{X}\to X\) be a functorial resolution of $X$ with the exceptional divisor \(E\). 
By Lemma \ref{lem-functorial-resolution-tangent}, there are generically surjective sheaf morphisms
\(\pi^{*} \mathcal{T}_{X} \to \mathcal{T}_{\widetilde{X}}(-\textup{log}\,E)\)
and \(\pi^{*} \mathcal{T}_{X} \to \mathcal{T}_{\widetilde{X}}\). 

Let $\nu \colon  \widetilde{C}  \to \widetilde{X}$ be a finite morphism from a smooth curve that intersects with $\pi^{-1}(X_{\sing})$. 
Set $C \coloneqq \pi(\nu(\widetilde{C} )) \subseteq X$. 
Then, we have a generically surjective sheaf morphism on $\widetilde{C}$
  $$
  (\nu^{*}\pi^{*} \mathcal{T}_{X}) / \Tor \rightarrow \nu^{*} \mathcal{T}_{\widetilde{X}}.
  $$
Since \(\mathcal{T}_X\) is strictly nef, so is $\nu^{*}\pi^{*} \mathcal{T}_{X}$.
Furthermore,  the same proof of \cite[Corollary 3.3]{LOY20} shows that $\nu^{*}\mathcal{T}_{\widetilde{X}}$ is also strictly nef.  
Indeed, for any quotient line bundle \(\nu^{*}\mathcal{T}_{\widetilde{X}}\to \mathcal{Q}\), 
since \( (\nu^{*}\pi^{*} \mathcal{T}_{X}) / \Tor \rightarrow \nu^{*} \mathcal{T}_{\widetilde{X}}\) is generically surjective, 
the composite map 
$$ (\nu^{*}\pi^{*} \mathcal{T}_{X}) / \Tor \rightarrow \nu^{*} \mathcal{T}_{\widetilde{X}}\to \mathcal{Q}$$ 
is non-zero. 
Thus, the image \(\mathcal{L}\) is a torsion-free and an invertible subsheaf of \(\mathcal{Q}\).
 Since \( (\nu^{*}\pi^{*} \mathcal{T}_{X}) / \Tor\) is strictly nef, so is \(\mathcal{L}\), which indicates that \(\deg(\mathcal{L})>0\).
 In particular, we obtain \(\deg(\mathcal{Q})\ge \deg(\mathcal{L})>0\).
Thus, we can conclude that  $\nu^{*}\mathcal{T}_{\widetilde{X}}$ is strictly nef by \cite[Proposition 2.1]{LiOY18}.

We claim that the variety \(\widetilde{X}\) is covered by rational curves. 
Suppose the contrary that \(\widetilde{X}\) is non-uniruled.
Then it follows from \cite[Corollary 0.3]{BDPP13} that \(K_{\widetilde{X}}\) is pseudo-effective.
Since \(\pi\) is the functorial resolution, \(T_{\widetilde{X}}\) is almost nef (cf.~Lemma \ref{lem-functorial-resolution-tangent}). 
By Corollary \ref{cor-c1-almostnef-tangent} (cf.~\cite{HIM22}), we see that \(T_{\widetilde{X}}\) is numerically flat and thus \(\widetilde{X}\) is an \'etale quotient of an abelian variety, a contradiction to the strict nefness of $\nu^{*}\mathcal{T}_{\widetilde{X}}$. 
Hence, the variety \(X\) is covered by rational curves. 
Let \(\ell\) be a rational curve not lying in \(E\) and let \(\mu \colon \widetilde{\ell}\cong\mathbb{P}^1\to\ell\) be the normalization.
Then, we have \(\mu^*\mathcal{T}_{\widetilde{X}}=\oplus\mathcal{O}(a_i)\) with each \(a_i\ge 1\).
On the other hand, the dual of the non-trivial morphism \(\mu^*\Omega_{\widetilde{X}}\to \Omega_{\widetilde{\ell}}\) gives rise to an injection \(\mathcal{O}(2)\to \oplus\mathcal{O}(a_i)\).
Therefore, there exists at least one \(a_i\) such that \(a_i\ge 2\), and thus \(\deg(-\mu^*K_{\widetilde{X}})\ge n+1\).
By  \cite[Corollary 0.4]{CMSB02}, we can conclude that \(\widetilde{X}\cong\mathbb{P}^n\).
Since the Picard number of 
$\widetilde{X}$ is 1, the birational image $X$ (which is also normal) has to be isomorphic to $\mathbb{P}^n$.  
\end{proof}

\begin{rem}[Varieties with big tangent sheaves]\label{rmk-v-big-singular}
According to \cite[Corollary 7.8]{FM21} and \cite[Corollary 1.3]{Iwa22b} (cf.\,\cite[Theorem B]{Wu21}), if a smooth projective variety $X$ has a {\it V-big} (or {\it big in the sense of Viehweg}) tangent bundle (i.e.,\,for any ample divisor $A$, there exists some \(m \in \Z_{+}\) such that $\Sym^{[m]}\mathcal{T}_{X} \otimes \mathcal{O}_{X}(-A)$ is pseudo-effective), then $X$  is isomorphic to $\mathbb{P}^n$. 

However, different from Proposition \ref{prop-strictly-nef-tangent}, 
when $X$ has singularities,  the same statement does not hold. 
Indeed, for any $m \in \mathbb{Z}{+}$, 
the weighted projective space $Y \coloneqq  \mathbb{P}(1,1,m)$ (with only quotient singularities) has V-big tangent sheaf.
To check this, we consider the finite morphism $\nu \colon  \mathbb{P}^{2} \to Y$ sending \([x_0:x_1:x_2]\) to \([x_0:x_1:x_2^m]\).
Since the differential morpshim $\mathcal{T}_{\mathbb{P}^2} \to \nu^{[*]}\mathcal{T}_{Y}$ is generically surjective, 
the sheaf $\nu^{[*]}\mathcal{T}_{Y}$ is V-big, and thus so is $\mathcal{T}_{Y}$  (cf.~\cite[Remark 1.3 (iv) and Lemma 1.4 (5)]{Vie83}).
Moreover, following the same argument as in Lemma \ref{lem-p-curved}, the tangent sheaf $\mathcal{T}_{Y}$ admits a singular hermitian metric $h$ such that $\sqrt{-1}\Theta_{h} \succeq \varepsilon \omega \otimes {\rm id}$, where $\omega$ is a K\"ahler form and $0 < \varepsilon \ll 1$. 
\end{rem}

\subsection{Surfaces and klt threefolds with almost nef tangent sheaf}

We conclude this paper by classifying minimal projective surfaces and klt threefolds with (almost) nef tangent sheaf. 
Here, the minimality of a projective \(\mathbb{Q}\)-Gorenstein variety \(X\) means that either the canonical divisor \(K_X\) is nef, 
or the \(K_X\)-negative contractions are of fiber type.
Let us first consider the surface case.

\begin{prop}
Let \(X\) be a normal projective surface with almost nef tangent sheaf.
Then, the variety \(X\) is one of the following$:$
\begin{enumerate}
\item[$(1)$] \(X\) is a $($possibly singular and not necessarily minimal$)$ rational surface. 
\item[$(2)$] \(X\) is a minimal ruled surface over an elliptic curve. 
\item[$(3)$] \(X\) is a cone over an elliptic curve; then the variety \(X\) is lc but not klt. 
\item[$(4)$] \(X\) is an \'etale quotient of an abelian surface.
\end{enumerate}
\end{prop}

\begin{proof}
Let \(\pi \colon Y\to X\) be the minimal resolution of $X$.  
By Lemma \ref{lem-functorial-resolution-tangent}.  
there exists a generically surjective morphism \(\pi^*\mathcal{T}_X\to \mathcal{T}_Y\) (cf.~\cite[Proposition 1.2]{Wah75}).
Then, the tangent sheaf \(\mathcal{T}_Y\) is almost nef. 
The fiber of \(Y\to X\) is free of \((-1)\)-curves, but we do not know whether \(Y\) itself is minimal or not. 

We first consider the case where  \(Y\) is non-uniruled.
Then,  by \cite[Theorem 1.2]{JLZ23}, the variety \(Y\) is minimal, and hence \(Y\) is an \'etale quotient of an abelian surface 
(see \cite[Proposition 4.11]{Iwa22}). 
By Lemma \ref{lem-bircontr-Q-abelian}, we obtain \(X\cong Y\).

We now consider the case where \(Y\) is uniruled. 
In the case \(q(Y)=0\), both \(Y\) and \(X\) are  rational surfaces. 
In the case \(q(Y)>0\), by \cite[Proposition 6.19]{DPS01} (cf.~\cite[Proposition 4.11]{Iwa22}), 
the variety \(Y\) is a minimal ruled surface over an elliptic curve.
We may assume that \(Y\to X\) is not isomorphic. 
Then, we obtain \(\rho(X)=1\) and \(Y\to X\) is the contraction of the minimal section (which always exists when the genus of the base curve is no more than one).
Therefore, the variety \(Y\) is a cone over an elliptic curve. In particular, it is not klt but lc.
\end{proof}

In the above classification result, it seems to be difficult to say deeper into the rational surface case 
because of the special blow-ups (cf.~Example \ref{ex-almost-nef-toric}). 
However, the nefness of the tangent sheaf leads to more restrictive geometric implications$:$

\begin{prop}\label{prop-classi-almost-threefold}
Let \(X\) be a normal projective surface with nef tangent sheaf.
Then, the variety \(X\) is one of the following$:$
\begin{enumerate}
\item[$(1)$] \(X\) is \(\mathbb{P}^2\) or \(\mathbb{P}^1\times\mathbb{P}^1\).
\item[$(2)$] \(X=\mathbb{P}(\mathcal{E})\) is a minimal ruled surface over an elliptic curve such that \(\mathcal{E}\) is semi-stable.
\item[$(3)$] \(X\) is a cone over an elliptic curve; then, the variety \(X\) is lc but not klt.
\item[$(4)$] \(X\) is a cone over a smooth rational curve. 
\item[$(5)$] \(X\) is an \'etale quotient of an abelian surface.
\end{enumerate}
\end{prop}
\begin{proof}
We first show that under the assumption of the nefness of \(\mathcal{T}_X\), 
the minimal resolution \(Y \to X\) has no \((-1)\)-curve. 
Indeed, if we suppose that there is a \((-1)\)-curve \(E\), 
the curve \(E\) does not lies in the fiber of \(\pi\). 
Hence, we obtain \(E=\pi_*^{-1}(\pi(E))\) and the induced morphism \(\pi|_E\) is a finite birational morphism. 
Then, there is a generically surjective morphism
\[
\pi^*\mathcal{T}_X|_E=(\pi|_E)^*(\mathcal{T}_X|_{\pi(E)})\to \mathcal{T}_Y|_E.
\]
Since \(\mathcal{T}_X\) is nef by assumption, the restriction \(\mathcal{T}_Y|_E\) is a nef vector bundle (cf.~Lemma \ref{lem-elementary-pro-nef} (3)).
However, this is absurd due to the following exact sequence
\[
0\to \mathcal{T}_E=\mathcal{O}_E(2)\to \mathcal{T}_Y|_E\to\mathcal{N}_{E/X}=\mathcal{O}_E(-1)\to 0.
\]

By applying \cite[Proposition 4.11]{Iwa22}, we see that \(Y\) is one of the following
\begin{enumerate}
\item[(a)] \(Y \cong \mathbb{P}^2\); 
\item[(b)] \(Y\) is a finite \'etale quotient of an abelian surface; 
\item[(c)]  $Y$ is a minimal ruled surface over either \(\mathbb{P}^1\) or an elliptic curve.
\end{enumerate}
For both cases (a) and (b), by Lemma \ref{lem-bircontr-Q-abelian} and the normality of $X$, we deduce that $X \cong Y$. 
In case (c), in the same way as above, the negative section (if exists)  must be contracted along $\pi$ by the nefness of $\mathcal{T}_X$. 
This indicates that either $\mathcal{T}_Y$ is nef and $Y \cong X$, or $\pi \colon Y \to X$ contracts the negative section, 
which implies that $X$ becomes a cone over either a smooth rational or an elliptic curve. 
Using the classification of smooth projective surfaces with nef tangent bundles in \cite[Theorem 3.1]{CP91}, our proposition is thus proved. 
\end{proof}

Finally, we come to the classification of projective klt threefolds with almost nef tangent sheaves.
We remind readers that, by Corollary \ref{cor-almost-nef}, in the case  (4) of the  proposition below, 
the variety \(X\) itself is indeed smooth.

\begin{prop}[{cf.~\cite{CP91}, \cite[Theorem 6.20]{DPS01}}]
Let \(X\) be a projective klt threefold with almost nef tangent sheaf \(\mathcal{T}_X\).
Then, after replacing \(X\) by a quasi-\'etale cover, we have the following classification according to the augmented irregularity \(\widehat{q}(X)\) of \(X\) (with \(\widehat{q}(X)\le 3\)).
\begin{enumerate}
\item[$(1)$] If \(\widehat{q}(X)=0\), then \(X\) is rationally connected.
\item[$(2)$] If \(\widehat{q}(X)=1\), then \(X\) admits a flat morphism onto an elliptic curve such that a very general fiber is rational and has almost nef tangent sheaf.
\item[$(3)$] If \(\widehat{q}(X)=2\), then \(X=\mathbb{P}_A(\mathcal{E})\) is an algebraic \(\mathbb{P}^1\)-bundle over an abelian surface \(A\).
\item[$(4)$] If \(\widehat{q}(X)=3\), then \(X\) is an abelian threefold.
\end{enumerate}
\end{prop}

\begin{proof}
We may assume that \(X\) is maximally quasi-\'etale.
Let \(\pi \colon \widetilde{X}\to X\) be a functorial resolution as in  Lemma \ref{lem-functorial-resolution-tangent}.
Then \(\mathcal{T}_{\widetilde{X}}\) is almost nef.
We consider the following commutative diagram of Albanese morphisms by noting that \(X\) has  only rational singularities 
\[
\xymatrix{
\widetilde{X}\ar[r]^\pi\ar[d]_f&X\ar[dl]^g\\
A.&
}
\]
By \cite[Corollary 4.7]{Iwa22}, the morphism \(f\) is a smooth fibration with rationally connected fibers. 
If \(\dim(A)=3\), then \(\widetilde{X}\cong X\cong A\) is an abelian threefold (cf.~Lemma \ref{lem-bircontr-Q-abelian}).
If \(\dim(A)=0\), then \(\widetilde{X}\) and hence \(X\) are rationally connected.

Suppose that \(\dim(A)=2\). 
Then \(\widetilde{X}\) is a \(\mathbb{P}^1\)-bundle over \(A\).
If \(\pi\) is not an isomorphism, by the projection formula, every fiber of \(f\) is contracted along \(\pi\), and in particular, we obtain \(\dim(X)\le 2\), which is absurd.
Hence, we see that \(\widetilde{X}\cong X\) is a \(\mathbb{P}^1\)-bundle over \(A\).
After a further \'etale cover of \(A\), we have \(X=\mathbb{P}_A(\mathcal{E})\) for some rank two vector bundle over \(A\) 
(see for example \cite[Lemma 7.4]{CP91}).

Finally, we suppose that \(\dim(A)=1\).
Then \(f\) is a smooth fibration over an elliptic curve with the fibers  having almost nef tangent bundles, too.
Since every curve contracted by \(\pi\) is also \(f\)-contracted, it follows that \(g\) is a flat morphism onto an elliptic curve such that a very general fiber is rational and has almost nef tangent sheaf (cf.~Theorem \ref{thm-almost-nef}).
\end{proof}

At the end of this paper, we present an open problem related to our results. 
\begin{prob}
Can we establish the same structure theorem for projective klt varieties with pseudo-effective tangent sheaf 
as in Theorems~\ref{thm-positively-curved} and \ref{thm-almost-nef}? 
\end{prob}

\bibliographystyle{alpha}

\end{document}